\documentclass[11pt]{amsart}
\usepackage{ifthen}

\newcommand\mtop{1in}
\newcommand\mbottom{1in}
\newcommand\mleft{1.2in}
\newcommand\mright{1.2in}
\usepackage{amssymb}
\usepackage{amsthm}
\usepackage{wasysym}
\usepackage{mathrsfs}
\usepackage[dvipsnames,svgnames,x11names]{xcolor}
\usepackage{hyperref}
\hypersetup{colorlinks=true,linkcolor=RoyalBlue,citecolor=RoyalBlue}
\usepackage[all]{xy}
\SelectTips{cm}{10}  
\usepackage{float}
\providecommand{\mparwidth}{1.0in}
\providecommand{\mtop}{1in}
\providecommand{\mbottom}{1in}
\providecommand{\mleft}{1.2in}
\providecommand{\mright}{1.2in}
\usepackage[top = \mtop, bottom = \mbottom, left = \mleft, right=\mright, marginparwidth=\mparwidth]{geometry}
\usepackage{fancyhdr}
\usepackage{setspace}
\usepackage{mathtools}
\usepackage{scalefnt}
\usepackage{microtype}
\usepackage{pifont}
\usepackage{ifthen}
\usepackage{makecell}
\usepackage{marginnote}
\usepackage{soul}

\usepackage{enumitem}
\setlist[enumerate,1]{leftmargin=6ex,topsep=-5pt}
\setlist[enumerate]{itemsep=7pt}
\setlist[itemize,1]{leftmargin=4ex,topsep=0em}
\setlist[itemize]{itemsep=5pt}
\setlist[itemize,2]{label=$\circ$}
\setlist[itemize,3]{label={\scalefont{0.6}\color{gray}$\blacktriangleright$}}
\setlist[itemize,4]{label=$\ast$}
\setlist{nolistsep}

\usepackage{mdframed}
\usepackage{tikz}
\usetikzlibrary{cd}
\usetikzlibrary{decorations.pathmorphing}
\usetikzlibrary{arrows, decorations.markings}

\usepackage[nolabel]{showlabels} 
\usepackage{longtable}
\usepackage{listings}
\begin{document}

\newcommand{\W}{\mathbb{W}}

\newcommand{\theoremnumstyle}{section}
\parskip=0.2in \parindent=0in
\allowdisplaybreaks 
\raggedbottom 

\newcommand*{\replacecommand}[1]{%
  \providecommand{#1}{}%
  \renewcommand{#1}%
}

\renewcommand{\arraystretch}{1.5}

\renewcommand{\l}{\overset}
\newcommand{\into}{\hookrightarrow}
\newcommand{\onto}{\twoheadrightarrow}
\newcommand{\tto}{\longrightarrow}
\newcommand{\too}[1]{\l{#1}\to}
\newcommand{\ttoo}[1]{\l{#1}\longrightarrow}
\newcommand{\intoo}[1]{\l{#1}\into}
\newcommand{\ontoo}[1]{\l{#1}\onto}
\newcommand{\mapstoo}[1]{\l{#1}\mapsto}
\newcommand{\bto}{\leftarrow}
\newcommand{\btto}{\longleftarrow}
\newcommand{\btoo}[1]{\l{#1}\bto}
\newcommand{\bttoo}[1]{\l{#1}\longleftarrow}
\newcommand{\binto}{\hookleftarrow}
\newcommand{\bonto}{\twoheadleftarrow}
\newcommand{\bintoo}[1]{\l{#1}\binto}
\newcommand{\bontoo}[1]{\l{#1}\bonto}
\newcommand{\ointo}{\hspace{3pt}\text{\raisebox{-1.5pt}{$\overset{\circ}{\vphantom{}\smash{\text{\raisebox{1.5pt}{$\into$}}}}$}}\hspace{3pt}}
\newcommand{\lu}{\underset}
\newcommand{\bimplies}{\impliedby}
\newcommand{\ints}{\cap}
\newcommand{\intss}{\bigcap}
\newcommand{\union}{\cup}
\newcommand{\unions}{\bigcup}
\newcommand{\djunion}{\sqcup}
\newcommand{\djunions}{\bigsqcup}
\newcommand{\propersubset}{\subsetneq}
\newcommand{\propersupset}{\supsetneq}
\newcommand{\contains}{\supseteq}
\newcommand{\semidirect}{\rtimes}
\newcommand{\isom}{\cong}
\newcommand{\normal}{\triangleleft}
\replacecommand{\dsum}{\oplus}
\newcommand{\dsums}{\bigoplus}
\newcommand{\tensor}{\otimes}
\newcommand{\tensors}{\bigotimes}
\newcommand{\cotensor}{{\,\scriptstyle\square}}
\let\originalbar\bar
\renewcommand{\bar}[1]{{\overline{#1}}}
\newcommand{\rlim}{\mathop{\varinjlim}\limits}
\newcommand{\llim}{\mathop{\varprojlim}\limits}
\newcommand{\x}{\times}
\replacecommand{\st}{\hspace{2pt} : \hspace{2pt}} 
\newcommand{\vv}{\vspace{10pt}}
\newcommand{\til}{\widetilde}
\renewcommand{\hat}{\widehat}
\newcommand{\hhat}{\wedge}
\newcommand{\iy}{\infty}
\newcommand{\hteq}{\simeq}
\newcommand{\dd}[2]{\frac{\partial #1}{\partial #2}}
\newcommand{\sm}{\wedge} 
\newcommand{\noqed}{\renewcommand{\qedsymbol}{}}
\newcommand{\adjoint}{\dashv}
\newcommand{\wreath}{\wr}
\newcommand{\heart}{\heartsuit}

\newcommand{\bigast}{\mathop{\vphantom{\sum}\mathchoice%
  {\vcenter{\hbox{\huge *}}}
  {\vcenter{\hbox{\Large *}}}{*}{*}}\displaylimits}

\renewcommand{\dim}{\operatorname{dim}}
\newcommand{\diam}{\operatorname{diam}}
\newcommand{\coker}{\operatorname{coker}}
\newcommand{\im}{\operatorname{im}}
\newcommand{\disc}{\operatorname{disc}}
\newcommand{\Pic}{\operatorname{Pic}}
\newcommand{\Der}{\operatorname{Der}}
\newcommand{\ord}{\operatorname{ord}}
\newcommand{\nil}{\operatorname{nil}}
\newcommand{\rad}{\operatorname{rad}}
\newcommand{\ssum}{\operatorname{sum}}
\newcommand{\codim}{\operatorname{codim}}
\newcommand{\cchar}{\operatorname{char}}
\newcommand{\sspan}{\operatorname{span}}
\newcommand{\rank}{\operatorname{rank}}
\newcommand{\Aut}{\operatorname{Aut}}
\newcommand{\Out}{\operatorname{Out}}
\newcommand{\Div}{\operatorname{Div}}
\newcommand{\Gal}{\operatorname{Gal}}
\newcommand{\Hom}{\operatorname{Hom}}
\newcommand{\Mor}{\operatorname{Mor}}
\newcommand{\Vect}{\operatorname{Vect}}
\newcommand{\Fun}{\operatorname{Fun}}
\newcommand{\Iso}{\operatorname{Iso}}
\newcommand{\Map}{\operatorname{Map}}
\newcommand{\Ho}{\operatorname{Ho}}
\newcommand{\Mod}{\operatorname{Mod}}
\newcommand{\Tot}{\operatorname{Tot}}
\newcommand{\cofib}{\operatorname{cofib}}
\newcommand{\fib}{\operatorname{fib}}
\newcommand{\hocofib}{\operatorname{hocofib}}
\newcommand{\hofib}{\operatorname{hofib}}
\newcommand{\Maps}{\operatorname{Maps}}
\newcommand{\Sym}{\operatorname{Sym}}
\newcommand{\Diff}{\operatorname{Diff}}
\newcommand{\Tr}{\operatorname{Tr}}
\newcommand{\Frac}{\operatorname{Frac}}
\renewcommand{\Re}{\operatorname{Re}} 
\renewcommand{\Im}{\operatorname{Im}}
\newcommand{\gr}{\operatorname{gr}}
\newcommand{\tr}{\operatorname{tr}}
\newcommand{\End}{\operatorname{End}}
\newcommand{\Mat}{\operatorname{Mat}}
\newcommand{\Proj}{\operatorname{Proj}}
\newcommand{\Th}{\operatorname{Thom}}
\newcommand{\Thom}{\operatorname{Thom}}
\newcommand{\Spec}{\operatorname{Spec}}
\newcommand{\Ext}{\operatorname{Ext}}
\newcommand{\Cotor}{\operatorname{Cotor}}
\newcommand{\Tor}{\operatorname{Tor}}
\newcommand{\vol}{\operatorname{vol}}

\newcommand{\Set}{\operatorname{Set}}
\newcommand{\Top}{\operatorname{Top}}
\newcommand{\Fin}{\operatorname{Fin}}
\newcommand{\Spaces}{\operatorname{Spaces}}
\newcommand{\Sp}{\operatorname{Sp}}
\newcommand{\Spectra}{\operatorname{Spectra}}
\newcommand{\Spt}{\operatorname{Spt}}
\newcommand{\Comod}{\operatorname{Comod}}
\newcommand{\Spf}{\operatorname{Spf}}
\newcommand{\tmf}{\mathit{tmf}}
\newcommand{\Tmf}{\mathit{Tmf}}
\newcommand{\TMF}{\mathit{TMF}}
\newcommand{\Null}{\operatorname{Null}}
\newcommand{\Fil}{\operatorname{Fil}}
\newcommand{\Sq}{\operatorname{Sq}}
\newcommand{\Stable}{\operatorname{Stable}}
\newcommand{\Poly}{\operatorname{Poly}}
\newcommand{\Cat}{\operatorname{Cat}}
\newcommand{\Orb}{\operatorname{Orb}}
\newcommand{\Exc}{\operatorname{Exc}}
\newcommand{\Part}{\operatorname{Part}}
\newcommand{\Comm}{\operatorname{Comm}}
\newcommand{\Res}{\operatorname{Res}}
\newcommand{\Thick}{\operatorname{Thick}}
\newcommand{\red}{{\operatorname{red}}}

\newcommand{\mmf}{\mathit{mmf}}
\newcommand{\Sm}{\operatorname{Sm}}
\newcommand{\Var}{\operatorname{Var}}
\newcommand{\Frob}{\operatorname{Frob}}
\newcommand{\Rep}{\operatorname{Rep}}
\newcommand{\Ch}{\operatorname{Ch}}
\newcommand{\Shv}{\operatorname{Shv}}
\newcommand{\Corr}{\operatorname{Corr}}
\newcommand{\Span}{\operatorname{span}}
\newcommand{\Sch}{\operatorname{Sch}}
\newcommand{\ev}{\operatorname{ev}}
\newcommand{\Homog}{\operatorname{Homog}}
\newcommand{\conn}{\operatorname{conn}}
\newcommand{\type}{\operatorname{type}}
\newcommand{\num}{\operatorname{num}}
\newcommand{\Aff}{\operatorname{Aff}}
\newcommand{\Psh}{\operatorname{Psh}}
\newcommand{\sk}{\operatorname{sk}}
\newcommand{\cosk}{\operatorname{cosk}}
\newcommand{\Cart}{\operatorname{Cart}}

\newcommand{\Br}{\operatorname{Br}}
\newcommand{\BW}{\operatorname{BW}}
\newcommand{\Cl}{\operatorname{Cl}}
\newcommand{\Conf}{\operatorname{Conf}}
\newcommand{\Alg}{\operatorname{Alg}}
\newcommand{\CAlg}{\operatorname{CAlg}}
\newcommand{\Lie}{\operatorname{Lie}}
\newcommand{\Coalg}{\operatorname{Coalg}}
\newcommand{\Ab}{\operatorname{Ab}}
\newcommand{\Ind}{\operatorname{Ind}}
\newcommand{\ind}{\operatorname{ind}}
\newcommand{\Fix}{\operatorname{Fix}}
\newcommand{\ho}{\operatorname{ho}}
\newcommand{\coeq}{\operatorname{coeq}}
\newcommand{\CMon}{\operatorname{CMon}}
\newcommand{\Sing}{\operatorname{Sing}}
\newcommand{\Inj}{\operatorname{Inj}}
\newcommand{\StMod}{\operatorname{StMod}}
\newcommand{\Loc}{\operatorname{Loc}}
\newcommand{\Free}{\operatorname{Free}}
\newcommand{\Art}{\operatorname{Art}}
\newcommand{\Gpd}{\operatorname{Gpd}}
\newcommand{\Def}{\operatorname{Def}}
\newcommand{\Hyp}{\operatorname{Hyp}}
\newcommand{\Pre}{\operatorname{Pre}}
\newcommand{\Lat}{\operatorname{Lat}}
\newcommand{\Coords}{\operatorname{Coords}}
\newcommand{\cone}{\operatorname{cone}}
\newcommand{\Spc}{\operatorname{Spc}}
\newcommand{\QCoh}{\operatorname{QCoh}}
\newcommand{\height}{\operatorname{ht}}
\newcommand{\Sub}{\operatorname{Sub}}
\newcommand{\Cone}{\operatorname{Cone}}
\newcommand{\Cocone}{\operatorname{Cocone}}
\newcommand{\Ran}{\operatorname{Ran}}
\newcommand{\Lan}{\operatorname{Lan}}
\newcommand{\LieAlg}{\operatorname{LieAlg}}
\newcommand{\Com}{\operatorname{Com}}
\newcommand{\CoAlg}{\operatorname{CoAlg}}
\newcommand{\Prim}{\operatorname{Prim}}
\newcommand{\Coh}{\operatorname{Coh}}
\newcommand{\FormalGrp}{\operatorname{FormalGrp}}
\newcommand{\Fact}{\operatorname{Fact}}
\renewcommand{\Bar}{\operatorname{Bar}}
\newcommand{\Cobar}{\operatorname{Cobar}}
\newcommand{\Ad}{\operatorname{Ad}}
\newcommand{\Moduli}{\operatorname{Moduli}}
\newcommand{\dgla}{\operatorname{dgla}}
\newcommand{\obl}{\operatorname{obl}}
\newcommand{\ob}{\operatorname{ob}}
\newcommand{\IndCoh}{\operatorname{IndCoh}}
\newcommand{\Cocomm}{\operatorname{Cocomm}}
\newcommand{\PreStk}{\operatorname{PreStk}}
\newcommand{\FormGrp}{\operatorname{FormGrp}}
\newcommand{\FormMod}{\operatorname{FormMod}}
\newcommand{\Grp}{\operatorname{Grp}}
\newcommand{\CommAlg}{\operatorname{CommAlg}}
\newcommand{\CoComm}{\operatorname{CoComm}}
\newcommand{\IndSch}{\operatorname{IndSch}}
\newcommand{\Dist}{\operatorname{Dist}}
\newcommand{\Triv}{\operatorname{Triv}}
\newcommand{\Oper}{\operatorname{Oper}}
\newcommand{\Bij}{\operatorname{Bij}}
\newcommand{\Syl}{\operatorname{Syl}}
\newcommand{\Inn}{\operatorname{Inn}}
\newcommand{\Emb}{\operatorname{Emb}}
\newcommand{\Gr}{\operatorname{Gr}}
\newcommand{\CRing}{\operatorname{CRing}}
\newcommand{\sSet}{\operatorname{sSet}}
\newcommand{\et}{\text{\'et}}
\newcommand{\Sh}{\operatorname{Sh}}
\newcommand{\Nil}{\operatorname{Nil}}
\newcommand{\Cech}{\v Cech}
\newcommand{\Stacks}{\operatorname{Stacks}}
\newcommand{\Pin}{\operatorname{Pin}}
\newcommand{\sgn}{\operatorname{sgn}}

\newcommand{\A}{\mathbb{A}}
\replacecommand{\C}{\mathbb{C}}
\newcommand{\CP}{\mathbb{C}\mathrm{P}}
\newcommand{\E}{\mathbb{E}}
\newcommand{\F}{\mathbb{F}}
\replacecommand{\G}{\mathbb{G}}
\renewcommand{\H}{\mathbb{H}}
\newcommand{\K}{\mathbb{K}}
\newcommand{\M}{\mathbb{M}}
\newcommand{\N}{\mathbb{N}}
\renewcommand{\P}{\mathbb{P}}
\newcommand{\Q}{\mathbb{Q}}
\newcommand{\R}{\mathbb{R}}
\newcommand{\RP}{\mathbb{R}\mathrm{P}}
\newcommand{\V}{\vee}
\newcommand{\T}{\mathbb{T}}
\providecommand{\U}{\mathscr{U}}
\newcommand{\Z}{\mathbb{Z}}
\newcommand{\e}{\mathfrak{g}}
\newcommand{\m}{\mathfrak{m}}
\newcommand{\n}{\mathfrak{n}}
\newcommand{\p}{\mathfrak{p}}
\newcommand{\q}{\mathfrak{q}}
\renewcommand{\t}{\mathfrak{t}}

\newcommand{\pa}[1]{\left( {#1} \right)}
\newcommand{\br}[1]{\left[ {#1} \right]}
\newcommand{\cu}[1]{\left\{ {#1} \right\}}
\newcommand{\ab}[1]{\left| {#1} \right|}
\newcommand{\an}[1]{\left\langle {#1}\right\rangle}
\newcommand{\fl}[1]{\left\lfloor {#1}\right\rfloor}
\newcommand{\ceil}[1]{\left\lceil {#1}\right\rceil}
\newcommand{\tf}[1]{{\textstyle{#1}}}
\newcommand{\patf}[1]{\pa{\textstyle{#1}}}

\renewcommand{\mp}{\ \raisebox{5pt}{\text{\rotatebox{180}{$\pm$}}}\ }
\renewcommand{\d}[1]{\ss \mathrm{d}#1}
\newcommand{\imod}{\hspace{-7pt}\pmod}

\renewcommand{\epsilon}{\varepsilon}
\renewcommand{\phi}{{\mathchoice{\raisebox{2pt}{\ensuremath\varphi}}{\raisebox{2pt}{\!\! \ensuremath\varphi}}{\raisebox{1pt}{\scriptsize$\varphi$}}{\varphi}}}
\newcommand{\ph}{{\color{white}.\!}}
\newcommand{\tspacer}{{\ensuremath{\color{white}\Big|\!}}}
\newcommand{\chii}{\raisebox{2pt}{\ensuremath\chi}}

\let\originalchi=\chi
\renewcommand{\chi}{{\!{\mathchoice{\raisebox{2pt}{
$\originalchi$}}{\!\raisebox{2pt}{
$\originalchi$}}{\raisebox{1pt}{\scriptsize$\originalchi$}}{\originalchi}}}}

\let\originalforall=\forall
\renewcommand{\forall}{\ \originalforall}

\let\originalexists=\exists
\renewcommand{\exists}{\ \originalexists}

\let\realcheck\check
\newcommand{\vH}{\realcheck{H}}

\newenvironment{qu}[2]
{\begin{list}{}
	  {\setlength\leftmargin{#1}
	  \setlength\rightmargin{#2}}
	  \item[]\footnotesize}
		  {\end{list}}

\newenvironment{titleblock}
{\begin{mdframed}[linecolor=black!20,backgroundcolor=black!15]\begin{center}}
{\end{center}\end{mdframed}}

\newenvironment{shadedblock}[1][5in]
{\bigskip\begin{mdframed}[align=center,userdefinedwidth=#1,linecolor=white,backgroundcolor=black!5]}{\end{mdframed}}

\newenvironment{shadedtitleblock}[2][5in]
{\begin{mdframed}[align=center,userdefinedwidth=#1,linecolor=white,backgroundcolor=black!15]\sc #2\end{mdframed}\begin{mdframed}[align=center,userdefinedwidth=#1,linecolor=white,backgroundcolor=black!5]}{\end{mdframed}}

\newcommand{\shadedheader}[1]{\vspace{25pt}\begin{mdframed}[linecolor=black!20,backgroundcolor=black!5]\sc #1\end{mdframed}\vspace{10pt}}


\newcommand{\itext}{\shortintertext} 
\makeatletter
\@ifundefined{resetu}{ 
	\renewcommand{\u}{\underbracket[0.7pt]} 
}
\makeatother
\makeatletter
\@ifundefined{resetO}{ 
	\renewcommand{\O}{\mathcal{O}}
}
\makeatother
\makeatletter
\@ifundefined{resetk}{ 
	\renewcommand{\k}{\Bbbk}
}
\makeatother

\makeatletter
\@ifundefined{mathds}{
	\newcommand{\Id}{Id}
	}{
	\newcommand{\Id}{\mathds{1}} 
	}
\@ifundefined{sethlcolor}{
	\newcommand{\fixmehl}[2]{\underline{#1}\marginpar{\raggedright\smaller\smaller #2}}
	}{\@ifundefined{marginnote}{
		\newcommand{\highlight}[1]{\ifmmode{\text{\sethlcolor{llgray}\hl{$#1$}}}\else{\sethlcolor{llred}\hl{#1}}\fi}
		\newcommand{\fixmehl}[2]{\highlight{#1}\marginpar{\raggedright\smaller\smaller\color{maroon} #2}}
		}{
		\newcommand{\highlight}[1]{\ifmmode{\text{\sethlcolor{llgray}\hl{$#1$}}}\else{\sethlcolor{llred}\hl{#1}}\fi}
		\newcommand{\fixmehl}[2]{\marginnote{\smaller \smaller\color{maroon} #2}{\highlight{#1}}}
		}
	}
\@ifundefined{xy}{}{
	\SelectTips{cm}{10}  
}
\@ifundefined{color}{}{
	\definecolor{darkgreen}{RGB}{0,70,0}
	\definecolor{dgreen}{RGB}{0,100,0}
	\definecolor{purple}{RGB}{120,00,120}
	\definecolor{gray}{RGB}{100,100,100}
	\definecolor{mgreen}{RGB}{0,150,0}
	\definecolor{dgreen}{RGB}{0,100,0}
	\definecolor{llgray}{RGB}{230,230,230}
	\definecolor{llred}{RGB}{237,228,228}
	\definecolor{lgreen}{RGB}{100,200,100}
	\definecolor{mgray}{RGB}{150,150,150}
	\definecolor{lgray}{RGB}{190,190,190}
	\definecolor{maroon}{RGB}{150,0,0}
	\definecolor{lblue}{RGB}{120,170,200}
	\definecolor{mblue}{RGB}{65,105,225}
	\definecolor{dblue}{RGB}{0,56,111}
	\definecolor{orange}{RGB}{255,165,0}
	\definecolor{brown}{RGB}{177,84,15}
	\definecolor{rose}{RGB}{135,0,52}
	\definecolor{gold}{RGB}{177,146,87}
	\definecolor{dred}{RGB}{135,19,19}
	\definecolor{mred}{RGB}{194,28,28}
	\newcommand{\edit}[1]{{\it{\color{gray}#1}}}
	\newcommand{\fixme}[1]{{\color{maroon}\it{#1}}}
	\newcommand{\citeme}[2][\!\!]{{\color{orange}[#2~\textit{#1}]}}
	\newcommand{\later}[1]{{\color{dgreen}#1}}
	\newcommand{\corr}[1]{{\color{red}\itshape #1}}
	\newcommand{\question}[1]{\itshape{\color{blue}#1}\upshape}
}
\@ifundefined{substack}{}{
    \newcommand{\attop}[1]{{\let\textstyle\scriptstyle\let\scriptstyle\scriptscriptstyle\substack{#1}}}
    \renewcommand{\atop}[1]{{\let\scriptstyle\textstyle\let\scriptscriptstyle\scriptstyle\substack{#1}}}
}
\makeatother

\newcommand{\tabentry}[1]{\renewcommand{\arraystretch}{1}\begin{tabular}{c}#1\end{tabular}}

\newcommand{\margin}[1]{\marginpar{\raggedright \scalefont{0.7}#1}} 

\newcommand{\pullback}{\ar@{}[rd]|<<{\text{\pigpenfont A}}}
\newcommand{\pushout}{\ar@{}[rd]|>>{\text{\pigpenfont I}}}

\newcommand{\longleftrightarrows}{\xymatrix@1@C=16pt{
\ar@<0.4ex>[r] & \ar@<0.4ex>[l]
}}
\newcommand{\longrightrightarrows}{\xymatrix@1@C=16pt{
\ar@<0.4ex>[r]\ar@<-0.4ex>[r] & 
}}
\newcommand{\mapstto}{\,\xymatrix@1@C=16pt{
\ar@{|->}[r] & 
}\,}
\newcommand{\mapsttoo}[1]{\xymatrix@1@C=16pt{
\ar@{|->}[r]^-{#1} & 
}}
\newcommand{\rightrightrightarrows}{\xymatrix@1@C=16pt{
\ar[r]\ar@<0.8ex>[r]\ar@<-0.8ex>[r] & 
}}
\newcommand{\longleftleftarrows}{\xymatrix@C=16pt{
 & \ar@<0.4ex>[l]\ar@<-0.4ex>[l]
}}
\newcommand{\leftleftleftarrows}{\xymatrix@1@C=16pt{
 & \ar[l]\ar@<0.8ex>[l]\ar@<-0.8ex>[l]
}}
\newcommand{\leftleftleftleftarrows}{\xymatrix@1@C=16pt{
 & \ar@<0.8ex>[l]\ar@<0.3ex>[l]\ar@<-0.3ex>[l]\ar@<-0.8ex>[l]
}}
\newcommand{\lcircle}{\ar@(ul,dl)} 
\newcommand{\rcircle}{\ar@(ur,dr)}
\newcommand{\intto}{\ \xymatrix@1@C=16pt{
\ar@{^(->}[r] & 
}}

\newcommand{\eva}[1]
   {{\color{Teal}\it #1}}
\newcommand{\EvaNote}[1]
   {\marginpar{\raggedright\smaller \smaller \color{Teal}#1}}

\newcommand{\rin}[1]
   {{\color{blue}\it #1}}

\makeatletter
\newcommand{\switchmargin}{
\if@reversemargin
\normalmarginpar
\else
\reversemarginpar
\fi
}
\makeatother
\newcommand{\highlighteva}[1]{\ifmmode{\text{\sethlcolor{Teal}\hl{$#1$}}}\else{\sethlcolor{LightCyan}\hl{#1}}\fi}
\newcommand{\EvaNoteHl}[2]{\marginnote{\smaller \smaller \color{Teal}
#2}{\highlighteva{#1}}\switchmargin}


\newtheoremstyle{gloss}{\topsep}{\topsep}{}{0pt}{\bfseries}{}{\newline}{\newline *{\bf #3} }
\theoremstyle{gloss}
\newtheorem*{defstar}{Definition}

\newtheoremstyle{newplain}{20pt}{0pt}{\it}{0pt}{\bfseries}{.}{1ex}{}
\theoremstyle{newplain}

\ifthenelse{\isundefined\theoremnumstyle}
	{\newtheorem{theorem}{Theorem}[section] 
	\numberwithin{equation}{section}} 
	{\ifthenelse{\equal\theoremnumstyle{}}
		{\newtheorem{theorem}{Theorem}}
		{\newtheorem{theorem}{Theorem}[section]
		\numberwithin{equation}{section}
		}
	}

\newtheorem{corollary}[theorem]{Corollary}
\newtheorem{claim}[theorem]{Claim}
\newtheorem{lemma}[theorem]{Lemma}
\newtheorem{proposition}[theorem]{Proposition}
\newtheorem{fact}[theorem]{Fact}

\newtheoremstyle{newtextthm}{20pt}{0pt}{}{0pt}{\bfseries}{.}{1ex}{}
\theoremstyle{newtextthm}
\newtheorem{definition}[theorem]{Definition}
\newtheorem{example}[theorem]{Example}
\newtheorem{problem}[theorem]{Problem}
\newtheorem{remark}[theorem]{Remark}
\newtheorem{notation}[theorem]{Notation}

\newtheorem*{theoremstar}{Theorem}
\newtheorem*{lemmastar}{Lemma}
\newtheorem*{corstar}{Corollary}
\newtheorem*{corollarystar}{Corollary}
\newtheorem*{propositionstar}{Proposition}
\newtheorem*{claimstar}{Claim}
\newtheorem*{examplestar}{Example}

\newcommand{\argforcustom}{}
\theoremstyle{newplain}
\newtheorem{helperforcustom}[theorem]{\argforcustom}
\newtheorem*{helperforcustomstar}{\argforcustom}
\newenvironment{custom}[1]{\renewcommand{\argforcustom}{#1}\begin{helperforcustom}}{\end{helperforcustom}}
\newenvironment{customstar}[1]{\renewcommand{\argforcustom}{#1}\begin{helperforcustomstar}}{\end{helperforcustomstar}}

\newenvironment{exercise}[1]{\hspace{1pt}\nn \large {\sc #1.}\hv \normalsize
\vspace{10pt}\\ }{} 
\newcommand{\subthing}[1]{\hv\large(#1)\hv\hv \normalsize }

\newcommand{\itemref}[1]{(\ref{#1})}

\renewcommand{\showlabelfont}{\tiny\tt\color{mgreen}}

\newcommand{\HH}{\hat{H}}
\newcommand{\SE}{{}^S\! E}
\newcommand{\pr}{\operatorname{pr}}
\newcommand{\wbar}{\originalbar{w}}

\title{Towards the $p=3$ Kervaire Invariant Problem: The $E_2$-page for the homotopy fixed points spectral sequence computing $\pi_*({E_6}^{hC_9})$}
\author{Eva Belmont}
\author{Rin Ray}
\maketitle

\begin{abstract}
Hill, Hopkins, and Ravenel suggest that the last remaining Kervaire invariant problem, the case of $p=3$, can be solved by computing the homotopy fixed points spectral sequence for $\pi_* E_6^{hC_9}$. We prove a detection theorem for this case and use a conjectural form of the $C_9$-action on $E_6$ to compute the $E_2$ page of this spectral sequence away from homological degree zero.
\end{abstract}

\section*{Introduction}

The classical Kervaire invariant one problem asks whether or not a framed manifold can be surgically converted to a sphere. Browder \cite{Browder} reduced this to a problem about survival of the elements $b_j := h_j^2 \in \Ext^{2,2^{j+1}}_{\mathcal{A}_2}(\F_2,\F_2)$ in the 2-primary Adams
spectral sequence for the sphere. Hill, Hopkins, and Ravenel's seminal work \cite{HHR} solved the Kervaire Invariant One problem at the prime 2 for all but finitely many cases by showing that $b_j$ does not survive for $j \geq 7$. The remaining unknown case ($j=6$) was solved recently by Lin, Wang, and Xu \cite{LWX}.

At an odd prime $p$, there is an analogous question about the $p$-primary Adams spectral sequence, called the odd-primary Kervaire invariant problem, which asks whether the $p$-fold Massey products $b_j = -\an{h_j,\dots,h_j} \in \Ext^{2,2(p-1)p^{j+1}}_{\mathcal{A}_p}(\F_p,\F_p)$ are permanent cycles. 
For $p\geq 5$ this had been solved decades earlier by Ravenel \cite{ravenel-arf}, who showed that none of the $b_j$'s for $j > 0$ are permanent cycles. However, his arguments fail at the prime 3: for all odd primes, he shows that the analogous classes $\beta_{p^j/p^j}$ for $j\geq 1$ in the Adams--Novikov spectral sequence support differentials of length $2p-1$, but at the prime 3, the argument relating non-survival of $\beta_{p^j/p^j}$ to the Adams classes $b_j$ fails. For example, $\beta_{9/9}$ supports a nontrivial differential but $b_2 = \Phi(\beta_{9/9}\pm \beta_7)$ (where $\Phi$ is the Thom reduction map from the Adams--Novikov to the Adams spectral sequence) is a permanent cycle.

More precisely, given a height $h$ and a finite subgroup $C_{p^k}$ of the Morava stabilizer group $\mathbb{S}_h$, there is a morphism of spectral sequences from the Adams--Novikov spectral sequence for the sphere to the homotopy fixed points spectral sequence (henceforth ``hfpss'') for the action of $C_{p^k}$ on the Lubin-Tate spectrum $E_h$.
$$ \xymatrix{
\Ext^{*,*}_{\mathcal{A}_p}(\F_p,\F_p)\ar@{=>}[d]^-{\text{ASS}} & \Ext^{*,*}_{BP_*BP}(BP_*, BP_*)\ar[l]_-\Phi\ar[r]\ar@{=>}[d]^-{\text{ANSS}} & H^*(C_{p^k}; \pi_*E_n)\ar@{=>}[d]^-{\text{hfpss}}
\\\pi_*(\mathbb{S})^\hhat_p & \pi_*(\mathbb{S})_{(p)}\ar[l]\ar[r] & \pi_*(E_h^{hC_{p^k}})
}$$
The strategy is as follows. If $b_j$ survived, then it would be detected by a
permanent cycle $x$ in the Adams--Novikov spectral sequence.
The goal is to show that the image of $x$ in $\pi_*(E_h^{ hC_{p^k}})$ is zero
for degree reasons, and then (the ``detection'' step) use this to deduce that
$x$ must be zero in $\pi_*(\mathbb{S})_{(p)}$, a contradiction. For $p\geq 5$,
this works for $k=1$, $h=p-1$, but at $p=3$, the detection step fails
with these values. Hill, Hopkins, and Ravenel assert that the strategy works for $p=3$ with $k=2$, $h=6$.

\begin{custom}{Conjecture}[Hill--Hopkins--Ravenel, \cite{HHR-odd}] \label{conj:HHR-odd}
~\begin{enumerate} 
\item (Detection) Every $x \in \Ext_{BP_*BP}^{2,4\cdot 3^{j+1}}$ such that $\Phi(x) = b_j$ has nontrivial image in $H^2(C_9; \pi_{4\cdot 3^{j+1}}E_6)$.
\item (Periodicity) $\pi_*E_6^{hC_9}$ is 972-periodic.
\item (Gap) $\pi_{-2}E_6^{hC_9} = 0$.
\end{enumerate}
\end{custom}
Since $x$ converges to a class in stem $4\cdot 3^{j+1}-2\equiv -2 \pmod {972}$ if $j \geq 4$, this would show that $b_j$ does not survive for $j\geq 4$.

A similar strategy is used in \cite{HHR} at $p=2$, but they replace $E_n^{hC_{p^k}}$ with a spectrum $\Omega$ that is related to a norm of the $C_2$-spectrum $MU_\R$. This strategy is not available at $p=3$ because there is currently no known odd-primary analogue of $MU_\R$. 
If such a spectrum ``$MU_{\mu_3}$'' were to be constructed, then there would be an alternate
version of Conjecture \ref{conj:HHR-odd} which replaces $E_6^{hC_9}$ by a norm
of $MU_{\mu_3}$. Hill, Hopkins, and Ravenel \cite{HHR-odd} have sketched the
proofs of the $MU_{\mu_3}$ versions of the detection, periodicity, and gap
conjectures under the assumption that $MU_{\mu_3}$ exists. We make no such
assumption, and instead observe below that their proof of the $MU_{\mu_3}$
detection theorem can be adapted to prove the $E_6$ version, Conjecture
\ref{conj:HHR-odd}(1). 

\begin{proposition}[$E_6$ Detection]\label{prop:detection0}
Conjecture \ref{conj:HHR-odd}(1) is true.
\end{proposition}

This is proved in Proposition \ref{prop:detection}.
Our overall strategy is to
prove the gap and periodicity theorems by computing the hfpss for $E_6^{hC_9}$, using the conjectural model for the $C_9$ action on $\pi_*E_6$ (Conjecture \ref{conj:action}). If one completely understood the map from the Adams--Novikov spectral sequence
$$ \Ext_{BP_*BP}^{*.*}(BP_*, BP_*) \to H^*(C_9; \pi_*E_6) $$
in addition to all of the hfpss differentials, this detection statement would be unnecessary. However, the comparison map is difficult to understand explicitly in terms of the Lubin--Tate coordinates; Proposition \ref{prop:detection0} is an alternative approach which reduces detection to computing the composite map to the simplest possible group cohomology $H^*(C_9; \bar{R}_*)$ (see Section \ref{sec:detection}).



The first step when computing a hfpss $H^*(G; \pi_*E)\implies \pi_*E^{hG}$ for a
$G$-spectrum $E$ is to understand the action of the group $G$ on the
coefficient ring $\pi_*E$.
The ring $\pi_0E_h \isom \W(\F_{p^h})[[u_1,\dots,u_{h-1}]]$
has a natural action of the Morava stabilizer group which comes from its
universal property: it carries the universal deformation of the Honda formal
group law $F_h$, and $\mathbb{S}_h$ is the group of automorphisms of
$F_h$ over $\F_{p^h}$ (see \cite{rezk-hopkins-miller}). However, this
description does not make it easy to write down the action explicitly in terms
of the generators $u_i$. In general, this is an open problem; see
\cite{nave} for some odd primary results at height $p-1$ modulo a large ideal,
and \cite{beaudry-models} for explicit results at all
heights for the prime 2.

However, a conjecture by the second author and independently by Hill, Hopkins, and Ravenel states that for any prime and any finite subgroup
$G\subseteq \mathbb{S}_h$, $\pi_0 E_h$ is isomorphic (as a $\W(\F_{p^h})[G]$-module)
to a ring with an easily describable $G$-action. Though the explicit isomorphism to the
presentation in terms of $BP(h)_* \simeq
\mathbb{W}(\F_{p^h})[[v_1,\dots,v_{h-1}]][u^{\pm}]$ is complicated, there is a natural $G$-equivariant isomorphism between the rings which allows us to work instead with a (graded induced reduced) regular representation; see Conjecture \ref{conj:action}. This conjecture is the subject of work in progress by the second author. In this paper,
we assume the conjecture and use that action to produce an explicit formula for
the $E_2$ page of the hfpss
$$  E_2 = H^*(C_9; \pi_* E_6) \implies \pi_* E_6^{hC_9}. $$
This is the first step in carrying out a variation of Hill, Hopkins, and Ravenel's program for
solving the $p=3$ Kervaire problem. Determining the higher differentials is work
in progress by the authors; this would solve the Periodicity and Gap conjectures.
We also mention
\cite{hill-EOn}, which assumes Hopkins' conjecture and provides a sketch for the hfpss for
$\pi_* E_6^{hC_3}$.


The formulation given below of the conjecture is due to the second author.

\begin{custom}{Conjecture} [Hill--Hopkins--Ravenel, Ray] \label{conj:action}
Let $p$ be any prime and let $\W := \W(\F_{p^{p^{k-1}(p-1)}})$.
Let $\bar{\rho} = \coker(\phi:\W \to \W[C_p])$ defined by
$\phi(1) = [1] + [\sigma] + \dots + [\sigma^{p-1}]$ for $C_p = \an{\sigma}$, where the generators of
$\W[C_p]$ have degree $-2$.
Then there is a $\W[C_{p^k}]$-module isomorphism
$$ \pi_*(E_{p^{k-1}(p-1)}) \isom
(\Sym(\Ind_{C_p}^{C_{p^k}}\bar{\rho})[\Delta^{-1}])^\hhat_I. $$
If $x_0$ denotes the image of $[1]\in \bar{\rho}$ in
$\Ind_{C_p}^{C_{p^k}}\bar{\rho}$ and $\gamma$ generates $C_{p^k}$, then
$\Delta = \prod_{i=0}^{p^k-1}\gamma^ix_0$ and $I = (p, x_0-\gamma x_0, \gamma x_0-\gamma^2 x_0,
\dots, \gamma^{p^k-1}x_0-\gamma^{p^k-2}x_0)$.
\end{custom}

\begin{remark} The reduced regular representation $\overline{\rho}$ of $C_p$ is usually written as the kernel of the augmentation ideal $\ker (\W[C_p] \to \W)$. This is isomorphic as a $\W[C_p]$-module to the formulation we use above, $\coker(\phi:\W \to \W[C_p])$. \end{remark} 

We will assume Conjecture \ref{conj:action} at $p=3$, $k=2$.
Write $M = \Sym(\Ind_{C_3}^{C^9}\bar{\rho})$.
We study the $E_2$ page of the hfpss
\begin{equation}\label{eq:hfpss} E_2^{n,t} =  H^n(C_9; \pi_tE_6) \isom H^n(C_9;
((\Delta^{-1} M)^\hhat_I)_t) \implies
\pi_{t-n}E_6^{C_9}\end{equation}
via a localized Serre spectral sequence (see Proposition \ref{lem:serre-tate})
\begin{equation}\label{eq:serre-intro} \SE^{p,q}_2 = H^p(C_9/C_3; \hat{H}^q(C_3;
(\Delta^{-1}M)^\hhat_I)) \implies \hat{H}^{p+q}(C_9; (\Delta^{-1}M)^\hhat_I). \end{equation}
Working with Tate cohomology significantly simplifies the problem while only
losing information in $H^0$, which is expected to not be needed for the Kervaire
problem.
\begin{proposition}\label{prop:serre-collapse}
The Serre spectral sequence \eqref{eq:serre-intro} collapses on the $E_4$ page.
The only differentials and extensions (see Lemma \ref{lem:d3}) come from
the Serre spectral sequence for trivial coefficients (Lemma \ref{lem:serre-triv}).
\end{proposition}



\textbf{Notation.} For a ring $R$, let $R[x_1,\dots,x_n]\tensor \Lambda_R[y_1,\dots,y_d]$ denote
the graded commutative algebra over $R$ on even generators $x_i$ and odd
generators $y_i$.

Our main theorem is as follows.
\begin{theorem}
Assume Conjecture \ref{conj:action} for $p=3$, $k=2$. Let $\k = \F_{3^6}$. 
As a module over $\til{S}_\k^*:= (\W(\k)/9)[b_1^\pm,s_3^\pm][[s_1,s_2]]$,
the $E_2$ page of the hfpss \eqref{eq:hfpss} is isomorphic in degrees $n>0$
to
$$ \til{S}_\k^*\{ 1,\bar{c},\delta,\bar{c}\delta,f_0,\dots,f_5,F_0,\dots,F_5 \} /
(3s_1,3s_2,3\delta,3\bar{c}\delta,3f_i,3F_i). $$
The degrees $(n,t)$ of the generators are as follows.
\begin{align*}
|s_1| & = (0,-6) & |s_2| & = (0,-12) & |s_3|  & = (0,-18) & |\delta| & = (0,-18)
\\|b_1| & = (2,0) & |\bar{c}| & = (3,-6) & |\bar{c}\delta| & = (3,-24)
\\|f_0| & = (1,-2) & |f_1| & = (1,-8) & |f_2| & = (1,-8) & |f_3| & = (1,-14) &
|f_4| & = (1,-14) & |f_5| & = (1,-20)
\\|F_0| & = (2,-4) & |F_1| & = (2,-10) & |F_2| & = (2,-10) &
|F_3| & = (2,-16) & |F_4| & = (2,-16) & |F_5| & = (2,-22)
\end{align*}
\end{theorem}


This is the first step towards proving the Periodicity and Gap theorems as
stated in Conjecture \ref{conj:HHR-odd}.

\textbf{Outline.} In Section \ref{sec:serre-triv} we do a preliminary
calculation and discuss a variant of the Serre spectral sequence that converges
to Tate cohomology. Then we turn to the main subject of the paper, the Serre spectral sequence
$\SE_r^{p,q}((\Delta^{-1}M)^\hhat_I)$
converging to the hfpss $E_2$ page for the spectrum $E_6$. In Section \ref{sec:p>0} we set up the spectral sequence
and compute $\SE_2^{p,q}(M)$ for $p>0$, and in Section \ref{sec:c1-d1} we compute
the more complicated $p=0$ part. These sections work with a simplified version
$M$ of $\pi_*E_6$ that is missing a localization and completion. In Section \ref{sec:completion} we 
relate $\SE_2^{p,q}((\Delta^{-1} M)^\hhat_I)$ to $\SE_2^{p,q}(M)$.
In Section \ref{sec:higher-serre} we consider
higher differentials and extensions in the Serre spectral sequence for
$\SE_2^{p,q}((\Delta^{-1}M)^\hhat_I)$ and state the main result (Theorem \ref{thm:main-completion}).
At the end of this section we discuss an $RO(C_9)$-graded version of the
hfpss. This extra grading does not have a major impact on the form of the hfpss $E_2$
page, but we expect them to become useful when calculating higher hfpss
differentials. Finally, in Section \ref{sec:detection} we prove the $E_6$ version of the detection theorem, Proposition \ref{prop:detection}.

\textbf{Acknowledgments.} 
	The first author is supported by the NSF grant DMS-2204357 and acknowledges the support and hospitality of the MPI Bonn and the Isaac Newton Institute for Mathematical Sciences, Cambridge under the Equivariant Homotopy in Context program, supported by EPSRC grant no EP/Z000580/1, and a Young Research Fellowship at the University of M\"unster. 
	The second author is supported by the Dynamics–Geometry–Structure group at Mathematics M\"unster which is funded by the DFG under Germany's Excellence Strategy EXC 2044–390685587, and acknowledges a visit to the MPI Bonn. We would like to thank Dominic Culver, Mingcong Zeng, and Tobias Barthel for helpful discussions during our stay at the MPI, Mike Hill for helpful discussions at the Isaac Newton Institute, and Paul Goerss for helpful discussions about the $I$-completion.

\section{Serre spectral sequence with trivial coefficients} \label{sec:serre-triv}
The main goal of this paper is to calculate $H^n(C_9; \pi_tE_6)$ for $n > 0$ using a Serre
spectral sequence for the extension $C_3 \to C_9 \to C_9/C_3$. In this section
we do a much simpler version of this calculation, where $\pi_*E_6$ is
replaced by the ground ring with trivial action. Experienced readers will
see this as an easy exercise, but we present it here so that we can use it for
comparison: the one family of Serre spectral sequence differentials in the $\pi_* E_6$ case comes
from the trivial module case (Lemma \ref{eq:coh-R}, Lemma \ref{eq:serre-sseq}). We also use this comparison to discuss localizing
the Serre spectral sequence so that it converges to Tate cohomology (see
Proposition \ref{lem:serre-tate}).

Recall that $\HH^i(G;M)=H^i(G;M)$ for $i>0$.

\begin{lemma}\label{lem:margolis}
Let $M$ be an abelian group. Then $H^0(C_n;M)=M^{C_n}$ and $\HH^*(C_n; M)$
is periodic over a generator $b$ of homological degree 2, with 
\begin{align*}
\HH^i(C_n; M) =\begin{cases} M^{C_n}/\Im(\tr) & i\text{ even}\\ \ker(\tr)/\Im(1-\gamma) & i\text{ odd} \end{cases}
\end{align*}
where $\tr:M\to M$ is multiplication by $1+\gamma+\dots+\gamma^{n-1}$ where $\gamma$ is the generator of $C_n$. In particular, if $M$ has trivial action then $H^0(C_n;M)=M$ and
\begin{align*}
\HH^i(C_n; M) =\begin{cases} M/n & i\text{ even}\\
\ker(M\too{\cdot\, n}M) & i\text{ odd}. \end{cases}
\end{align*}
If $M$ has trivial action and is torsion-free, then $\HH^*(C_n;M) =
(M/n)[b^\pm]$.
\end{lemma}
\begin{proof}
This is standard; it comes from the periodic resolution
\begin{equation}
\label{eq:margolis} 0\bto M\bto M\tensor_\Z \Z[C_n]\btoo{1-\gamma} M\tensor_\Z
\Z[C_n]\btoo{\tr} \dots.\qedhere
\end{equation}
\end{proof}

Lemma \ref{lem:margolis} directly gives $H^*(C_9; R)$, but we
calculate the Serre spectral sequence so that we can use the differential and
extension for the $\pi_*E_6$ case.
\begin{lemma}\label{lem:serre-triv}
Let $R$ be a torsion-free commutative ring which we give a trivial $C_9$-action. The Serre spectral
sequence
\begin{equation}\label{eq:serre-triv} \SE_2^{p,q}(R) =  H^p(C_9/C_3; H^q(C_3;R))\implies H^{p+q}(C_9; R) \end{equation}
has $E_2$ page
$$ \SE_2^{*,*}(R) \isom \big(R[b_1,b_2]/(3b_1,3b_2)\big)\{ 1, a_2b_1 \} $$
as a module over $H^*(C_9/C_3; R) \isom R[b_2]/(3b_2)$, where $|b_1| =
(0,2)$, $|b_2| = (2,0)$, and $|a_2b_1| = (1,2)$ (which is indecomposable).
There is a differential
$$ d_3(a_2b_1) = b_2^2 $$
and a hidden extension $3\cdot b_1 = b_2$.
\end{lemma}

\begin{proof}
Let $k = R/3$. Since $R$ is torsion-free, Lemma \ref{lem:margolis} gives
\begin{align}
\label{eq:coh-R}\HH^*(C_9; R) & = (R/9)[b^\pm]
\\\notag\HH^*(C_9; k) & = k[a,b^\pm]/(a^2)
\\\notag \HH^*(C_3;R) & = k[b^\pm]
\\\notag \HH^*(C_3;k) & = k[a,b^\pm]/(a^2)
\\\notag H^0(C_9;R) & = R
\end{align}
with $|a| = 1$ and $|b| = 2$.

In $\SE_2^{*,*}(R)$,
write $b_1$ for the generator of the inner cohomology $\hat{H}^*(C_3; R)$ and
$a_2$ (if applicable) and $b_2$ for the generators of $H^*(C_9/C_3; M)$ for $M=R$ or $k$. The Serre $E_2$ page is shown below, along with differentials which we explain below.

\begin{figure}[H]
\centerline{\includegraphics[width=180pt]{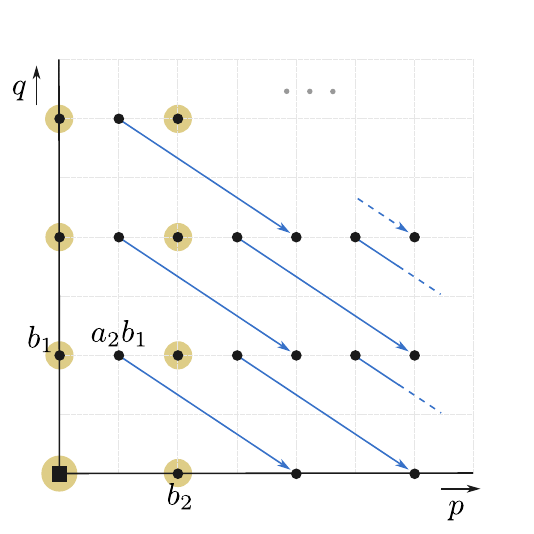}}
\caption{The Serre spectral sequence \eqref{eq:serre-triv}. Dots represent copies of $k$, and the square represents a copy of $R$. Surviving classes are highlighted.}
\label{fig:serre-triv}
\end{figure}

The only possible differential pattern that produces the correct
$C_9$-cohomology (as shown in \eqref{eq:coh-R}) is the one given. For degree
reasons, no higher differentials are possible.
The hidden extension is forced by the fact that $H^2(C_9;R)=\HH^2(C_9;R)=R/9$.
\end{proof}

\begin{remark}
One might expect there to be a convergent spectral sequence
$$ \HH^p(C_9/C_3; \HH^q(C_3;R))\implies \HH^{p+q}(C_9; R), $$
which corresponds to inverting $b_1$ and $b_2$ in the calculation above. But $E_4 = 0$ after inverting $b_1$ and $b_2$, so the spectral sequence does not converge.
\end{remark}

However, if one just inverts $b_1$, then the spectral sequence does converge. This occurs in general.
\begin{proposition}\label{lem:serre-tate}
Let $M$ be an $R[C_9]$-module.
There is a spectral sequence
\begin{equation}\label{eq:serre-M} \SE_2^{p,q}(M) = H^p(C_9/C_3; \hat{H}^q(C_3; M)) \implies \hat{H}^{p+q}(C_9;M).
\end{equation}
that converges if $\SE_r^{*,*}(M)$ for some $r$ is concentrated in finitely many $p$-degrees.
\end{proposition}
\begin{proof}
Since $M$ is $C_9$-equivariantly an $R$-module, $H^*(C_9; M)$ is an $H^*(C_9; R)$-module and $H^*(C_3; M)$ is an $H^*(C_3;R)$-module. The periodicity operator on $H^*(C_3; M)$ comes from the action of the periodicity element $b\in H^2(C_3; R)$ which converges to the periodicity operator $b\in H^*(C_9; R)$.
Thus the periodicity class for $H^*(C_9, M)$ is detected by the periodicity
class for $H^*(C_3, M)$.
Now use the fact that Tate cohomology for a cyclic group is obtained by inverting the periodicity class $b$, and this localization commutes with $H^*(C_9/C_3, -)$.

For convergence, we follow \cite[Proposition 4.4.1 proof]{eva-thesis}, which says there are two things that can go wrong with convergence of a localized spectral sequence $b^{-1}E_r \Rightarrow b^{-1}H^*$: (1) a $b$-tower of permanent cycles is not in $b^{-1}E_\infty$ because the tower is split into infinitely many pieces in the spectral sequence, connected by hidden $b$-multiplications; (2) a $b$-periodic tower supports a differential to an infinite sequence of torsion elements, and hence this differential is not recorded in $b^{-1}E_r$.

\begin{figure}[H]
\centerline{\includegraphics[width=140pt]{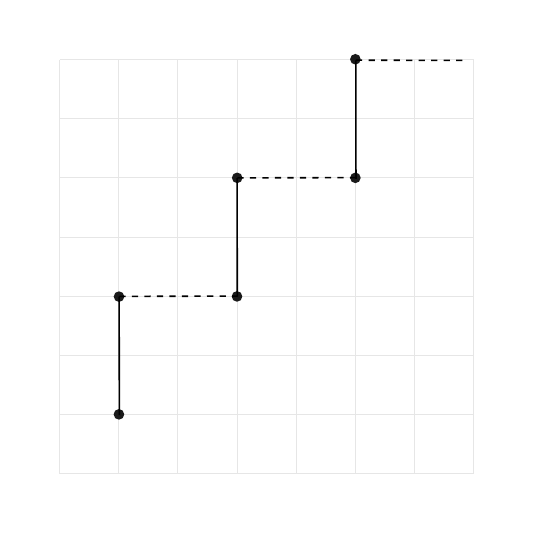} \hspace{30pt}
\includegraphics[width=140pt]{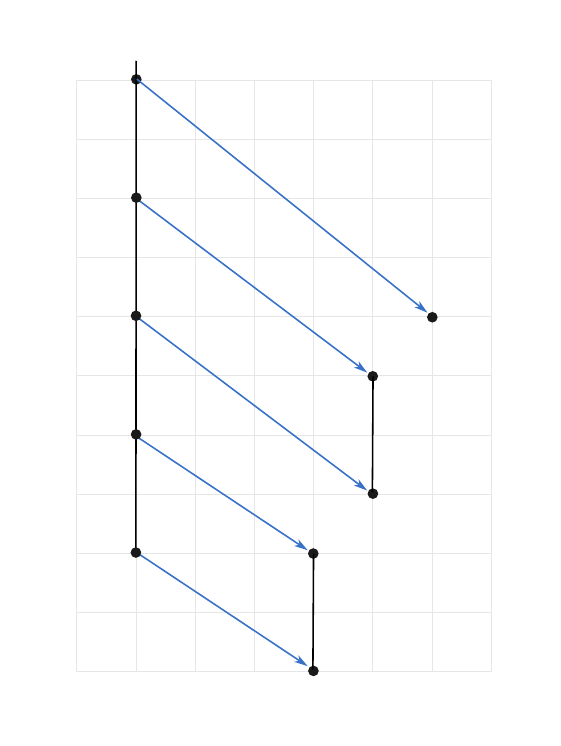}}
\caption{Examples illustrating the two possible issues with convergence. Dotted lines in the first picture indicate hidden extensions by $b$. Blue arrows in the second picture indicate differentials.}
\end{figure}

Since $|b| = (0,2)$, a hidden $b$-extension increases degree by $(n, 2-n)$ where $n$ is the filtration jump. So an infinite tower of $b$-extensions would have unbounded $p$-degree. The second scenario requires a sequence of differentials of increasing length with source a $b$-tower $\{ xb^i \}$. Since differentials increase degree by $(p+r, q-r+1)$, this also requires a sequence of elements in arbitrarily large $p$-degrees.
\end{proof}

\section{Serre $E_2^{p,q}$ for $p > 0$}\label{sec:p>0}
The overall goal is to compute $\hat{H}^*(C_9; (M[\Delta^{-1}])^\hhat_I)$
where $$ M = \Sym(\Ind^{C_9}_{C_3}\bar{\rho}), $$
but we postpone the consideration of the completion and $\Delta$-localization
until Section \ref{sec:completion}.
Following Lemma \ref{lem:serre-tate}, consider the Tate cohomology Serre spectral sequence
\begin{equation}\label{eq:serre} E_2^{p,q} = H^p(C_9/C_3; \hat{H}^q(C_3; M)) \implies
\hat{H}^{p+q}(C_9; M). \end{equation}
We will show convergence in the proof of Theorem \ref{eq:E2-sec5}.
In this section, we discuss the structure of $\hat{H}^*(C_3; M)$ as a
$C_9/C_3$-module and determine $E_2^{p,q}$ for $p>0$
(Proposition \ref{prop:p>0}) by replacing $H^p$ with $\hat{H}^p$.
The $p=0$ part is more complicated, and
will be treated in Section \ref{sec:c1-d1}.
We start by rewriting Conjecture \ref{conj:action} in a more convenient form.


\begin{custom}{Notation}
For the rest of the paper, let $\k = \F_{3^6}$ and $\W = \W(\k)$.
\end{custom}

\begin{corollary} \label{cor:conjecture}
Conjecture \ref{conj:action} implies that
there is a $\W[C_9]$-module isomorphism
$$ \pi_* E_6 \isom (\Sym(\Ind_{C_3}^{C_9}\bar{\rho})[\Delta^{-1}])^\hhat_I $$
where
$$ \Sym(\Ind_{C_3}^{C_9}\bar{\rho}) \isom
\W[x_0,\dots,x_8]\Big/\left(\atop{x_0+x_3+x_6,\\x_1+x_4+x_7,\\x_2+x_5+x_8}\right), $$
$I = (3,x_0-x_1,\dots,x_7-x_8)$, $\Delta = \prod_{i=0}^8 x_i$, and the
generator $\gamma\in C_9$ acts by $\gamma(x_i)=x_{i+1}$ (where subscripts are
taken mod 9).
\end{corollary}
\begin{proof}
Comparing to the notation in Conjecture \ref{conj:action}, here we let $x_i =
\gamma^i x_0$ where $C_9 = \an{\gamma}$, and observe that $\sigma = \gamma^3$.
Let $y_0 = x_0+x_3+x_6$, $y_1=x_1+x_4+x_7$, and $y_2 = x_2+x_5+x_8$.
Since $\bar{\rho} \isom \W\{ x_0,x_3,x_6 \}/y_0\W\{ x_0,x_3,x_6 \}$, we have 
$$ \Ind_{C_3}^{C_9} \bar{\rho} \isom \coker(\W\{ y_0,y_1,y_2 \} \to \W\{ x_0,\dots,x_8 \})$$
where $\gamma$ cyclically permutes
the generators $x_i$ and $C_3 = \an{\gamma^3} \subseteq C_9$. The kernel of
the quotient
$\W[ x_0,\dots,x_8] \to \Sym(\Ind^{C_9}_{C_3}\bar{\rho})$ is the ideal
$(y_0,y_1,y_2)$ of $\W[x_0,\dots,x_8]$.
\end{proof}
 
In order to analyze $\hat{H}^*(C_3; M)$ in \eqref{eq:serre}, we will need the following two
lemmas.

\begin{lemma}\label{lem:C3-kunneth}
Let $M_1$ and $M_2$ be $R[C_3]$-modules for $R = \k$ or $\W$. If $R = \W$,
assume that $M_1$ and $M_2$ are torsion-free. Then
$$ \hat{H}^*(C_3; M_1\tensor M_2) \isom \hat{H}^*(C_3;
M_1)\tensor_{\hat{H}^*(C_3;k)} \hat{H}^*(C_3; M_2). $$
\end{lemma}
\begin{proof}
	Since $\k$ has characteristic 3,
the indecomposable $\k[C_3]$-modules are the trivial module $\k$, the free
module $\k[C_3]$, and the indecomposable 2-dimensional module $\k^\sigma = \k\{ 1,\gamma,\gamma^2 \}/(1+\gamma+\gamma^2)$.
In the case $R = \W$, the indecomposable \emph{torsion-free} modules have the
same form (trivial, free, or $R^\sigma = R\{ 1,\gamma,\gamma^2 \}/(1+\gamma+\gamma^2)$) \cite{reiner}.
In both cases, short
exact sequence $0\to R\{ 1+\gamma+\gamma^2 \} \to R[C_3]\to R^\sigma\to 0$
shows that $R^\sigma \isom \Sigma R$ in the stable module category. So $M_1\tensor
M_2$ decomposes as a direct sum of shifts of the trivial module, and we can use the fact that group cohomology of finite groups commutes with arbitrary direct sums.
\end{proof}
\begin{remark}
This is not true if Tate cohomology is replaced by ordinary cohomology; for
example, take $M_1 = M_2 = \k[C_3]$.
\end{remark}

\begin{lemma}\label{lem:coh-rhobar}
Suppose $R$ is a ring with no 3-power torsion and suppose $C_3$ acts on $R[t_0,t_1,t_2]$ by permuting the
polynomial generators. Then
\begin{align*}
\hat{H}^*(C_3; R[ t_0,t_1,t_2])  & \isom (R/3)[b^\pm, d]
\\\hat{H}^*(C_3; R[ t_0,t_1,t_2]/(t_0+t_1+t_2))  & \isom (R/3)[b^\pm,d]
\tensor_R \Lambda_R[c]
\end{align*}
where $d = t_0t_1t_2$ has degree $(0,3n)$ in the $(s,t)$ grading, and $|c| =
(1,n)$ where $|t_i| = n$.
\end{lemma}
\begin{proof}
The first statement comes from grouping
the monomial basis for $R[t_0,t_1,t_2]$ into $R[C_3]$-summands $\dsums_{i\geq
0}R\{ (t_0t_1t_2)^i \}\dsum F$ where $F$ is free and the powers of $d =
t_0t_1t_2$ have trivial action.
Thus
$\hat{H}^*(C_3; R[t_0,t_1,t_2]) \isom \hat{H}^*(C_3; R)[d]$. For the
second statement, use the short exact sequence
\begin{align*} 0 &\to (t_0+t_1+t_2) \to R[t_0+t_1+t_2] \to
R[t_0+t_1+t_2]/(t_0+t_1+t_2)\to 0. \qedhere \end{align*}
More details can be found in \cite[\S10.1]{stojanoska-duality}. 
\end{proof}

\begin{lemma}\label{lem:inner-coh}
As an algebra over $\hat{H}^*(C_3; \W) = \k[b_1^\pm]$ we have
$$\hat{H}^*(C_3; M) \isom \k[b_1^\pm,d_1,d_2,d_3]\tensor
\Lambda_\k[c_1,c_2,c_3]$$
where $d_1 = x_0x_3x_6$, $d_2 = x_1x_4x_7$, $d_3 = x_2x_5x_8$ are in $\hat{H}^0(C_3; M)$, and $c_1,c_2,c_3$ are in $\hat{H}^1(C_3; M)$.
As a module over $\k[C_9/C_3]\isom \k[C_3]$, it has a decomposition
\begin{align*}
 & \Big(\dsums_\attop{\epsilon\in \{ 0,1 \}\\i\geq 0} \k\{ (c_1c_2c_3)^\epsilon(d_1d_2d_3)^i\}\dsum
 \dsums_\attop{\epsilon\in\{ 0,1 \},\,i,j\geq 0\\(i,j)\neq (0,0)}\!\!\! \k[C_3]\{(c_1c_2c_3)^\epsilon d_1^i d_2^i d_3^j \}
				\dsums_\attop{\epsilon\in \{ 0,1 \}\\i<j<k\\\text{or } i >
				j > k} \k[C_3]\{ (c_1c_2c_3)^\epsilon d_1^id_2^jd_3^k \}
				\\  & \dsum\dsums_\attop{\ell\in \{ 1,2,3 \}\\i\leq j<k\\\text{or } i \geq j
				> k} \k[C_3]\{ c_\ell d_1^id_2^jd_3^k \}\dsum
				\dsums_\attop{(\ell,m)\in \{ (1,2),(2,3),(3,1) \}\\i\leq j<k\text{ or } i \geq j
			> k} \k[C_3]\{ c_\ell c_m d_1^id_2^jd_3^k \}\Big) \tensor
			\k[b_1^\pm]
\end{align*}
where the first term consists of trivial summands and the other terms are free
summands.
\end{lemma}
\begin{proof}
If $y_0 = x_0+x_3+x_6$, $y_1 = x_1+x_4+x_7$, and $y_2 = x_2+x_5+x_8$
we have an isomorphism of $\W[C_3]$-modules
$$ M\isom \W\{ x_0,x_3,x_6 \}/(y_0) \tensor \W\{ x_1,x_4,x_7 \}/(y_1) \tensor \W\{ x_2,x_5,x_8 \}/(y_2). $$
Using Lemmas \ref{lem:C3-kunneth} and \ref{lem:coh-rhobar}, we have
\begin{align*}
\hat{H}^*(C_3; M) & \isom \hat{H}^*\Big(C_3; {\W\{ x_0,x_3,x_6 \}\over
(y_0)}\Big)
\tensor\hat{H}^*\Big(C_3; {\W\{ x_1,x_4,x_7 \}\over (y_1)}\Big)
\tensor\hat{H}^*\Big(C_3; {\W\{ x_2,x_5,x_8 \}\over (y_2)}\Big)
\\ & \isom \k[b_1^\pm,d_1,d_2,d_3] \tensor \Lambda_\k[c_1,c_2,c_3]
\end{align*}
where $d_1 = x_0x_3x_6$, $d_2 = x_1x_4x_7$, $d_3 = x_2x_5x_8$, and $c_1,c_2,c_3$ are the
corresponding homological degree-1 classes. The generator of $C_9/C_3$
permutes the $c_i$'s and permutes the $d_i$'s. The module decomposition follows
by considering the orbit generated by each monomial.
\end{proof}

\begin{proposition}\label{prop:p>0}
We have
$$\hat{H}^*(C_9/C_3; \hat{H}^*(C_3; \Sym(\Ind^{C_9}_{C_3}\bar{\rho}))) \isom
\k[b_1^\pm,b_2^\pm, s_3]\tensor \Lambda[a_2,\bar{c}]$$
where $\bar{c} = c_1c_2c_3$ and $s_3 = d_1d_2d_3$.
\end{proposition}
In this statement, the inner cohomology is a module over $\hat{H}^*(C_3; \W(\k)) =
\k[b_1^\pm]$ and the outer cohomology is a module over
$\hat{H}^*(C_9/C_3; \k) = \k[b_2^\pm] \tensor \Lambda_\k[a_2]$.
Note that $s_3 = \Delta$ from Corollary \ref{cor:conjecture}.
\begin{proof}
	In order to determine the $C_9/C_3$-Tate cohomology of $\hat{H}^*(C_3; M)$ we just
need to know the non-free summands. By Lemma \ref{lem:inner-coh} these are $\k\{ s_3^i \}$ and
$\k\{ \bar{c}s_3^i \}$ where $\bar{c}=c_1c_2c_3$ and $s_3=d_1d_2d_3$; these all have trivial
$C_9/C_3$-action. Thus we have
$$ \hat{H}^*(C_9/C_3; \hat{H}^*(C_3; M)) \isom \k[b_2^\pm] \tensor \Lambda_\k[a_2]
\tensor \k[b_1^\pm, s_3]\{ 1,\bar{c} \} $$
where $\hat{H}^*(C_9/C_3; \k) \isom \k[b_2^\pm]\tensor \Lambda_\k[a_2]$.
The multiplicative structure results from observing that $\bar{c}^2 = 0$.
\end{proof}

\section{Serre $E_2^{p,q}$ for $p=0$} \label{sec:c1-d1}
In this section we compute the most complicated part of the Serre spectral
sequence $E_2$ term
$\SE^{*,*}(M)$ for $M = \Sym(\Ind_{C_3}^{C_9}\bar{\rho})$,
namely
\begin{equation}\label{eq:E20} \SE_2^{0,*} = H^0(C_9/C_3; \hat{H}^*(C_3;
\Sym(\Ind^{C_9}_{C_3}\bar{\rho}))). \end{equation}
By Lemma \ref{lem:inner-coh}, this is 
$\big( \k[b_1^\pm,d_1,d_2,d_3]\tensor \Lambda_\k[c_1,c_2,c_3]\big)^{C_3}$,
the $C_3$-invariants of the graded commutative ring on the given generators,
where $|d_i| = 0$, $|c_i| = 1$, and $|b_1| = 2$ (this is grading by $q$ in the sense of $\SE_2^{p,q}$).
The action of $C_3$ permutes $\{c_1,c_2,c_3\}$ and $\{d_1,d_2,d_3\}$
and fixes $b_1$, so it suffices to compute $T^{C_3}$ where
$$ T := \k[d_1,d_2,d_3]\tensor \Lambda_\k[c_1,c_2,c_3]. $$

The polynomial invariants of the symmetric group are generated by the elementary symmetric
polynomials. For the alternating group (such as $A_3 = C_3$), one need only add
one generator $\delta = \prod_{i<j}(x_i-x_j)$; see for example \cite[Appendix B, (iii)]{benson-invariants}.
\begin{lemma}\label{lem:alternating}
Suppose the symmetric group $S_3$ acts on $\k[d_1,d_2,d_3]$ by permuting the
generators and $C_3 \subseteq S_3$ is the subgroup of cyclic permutations.
Then
\begin{align*}
S:=\k[d_1,d_2,d_3]^{S_3}  & \isom \k[s_1,s_2,s_3]
\\\k[d_1,d_2,d_3]^{C_3} & \isom \k[s_1,s_2,s_3,\delta]/(\delta^2+s_2^3+s_1^3s_3-s_1^2s_2^2)
\end{align*}
where $s_1 = d_1+d_2+d_3$, $s_2 = d_1d_2 + d_2d_3 + d_3d_1$, $s_3 = d_1d_2d_3$,
and $\delta = (d_1-d_2)(d_2-d_3)(d_3-d_1)$.
\end{lemma}

In particular, we can apply this to the subalgebra $\k[d_1,d_2,d_3]\subseteq
\k[d_1,d_2,d_3]\tensor \Lambda_{\k}[c_1,c_2,c_3]$. Next
we will get a (non-optimal) set of multiplicative generators for the desired
ring of $C_3$-invariants.

\begin{definition}
If $X$ is a monomial that is not fixed by $C_3$, write $\tr(X) = X + \gamma X +
\gamma^2 X$ where $C_3 = \an{\gamma}$. 
\end{definition}

\begin{lemma} \label{lem:c-gens}
Let $T = \k[d_1,d_2,d_3]\tensor \Lambda_{\k}[c_1,c_2,c_3]$ as above.
Then $T^{C_3}$ is generated as a module over $S:= \k[d_1,d_2,d_3]^{S_3}$ by
\begin{align}
\label{eq:c-gens}  & 1, \delta, \tr(c_1), \tr(c_1c_2), c_1c_2c_3, c_1c_2c_3\delta,
\\\notag & \tr(c_id_1), \tr(c_id_1d_2), \tr(c_id_1^2d_2), \tr(c_id_1d_2^2)
\\\notag & \tr(c_ic_jd_1), \tr(c_ic_jd_1d_2), \tr(c_ic_jd_1^2d_2), \tr(c_ic_jd_1d_2^2)
\end{align}
for $i=1,2,3$, and $(i,j)\in \{ (1,2),(2,3),(3,1) \}$.
\end{lemma}
\begin{proof}
Note that every monomial can be written in the form $X$, $c_iX$, $c_ic_jX$, or
$c_1c_2c_3X$ for $X\in \k[d_1,d_2,d_3]$ and $i\neq j$.
For ease of notation, write $c_1[i_1,i_2,i_3] := c_1d_1^{i_1}d_2^{i_2}d_3^{i_3}$ and
similarly for other $c$-multiples.
The orbit types are as follows:
\begin{enumerate} 
\item $[a,a,a]$ and $c_1c_2c_3[a,a,a]$, which are fixed by $C_3$;
\item the $C_3$-free orbits generated by $c_i[a,a,a]$ and $c_ic_j[a,a,a]$;
\item the $C_3$-free orbits generated by $[a,a,b]$ and $c_1c_2c_3[a,a,b]$ with
$a\neq b$;
\item the $C_3$-free orbits generated by $[i_1,i_2,i_3]$ and
$c_1c_2c_3[i_1,i_2,i_3]$ where $i_1,i_2,i_3$ are all distinct;
\item the $C_3$-free orbits generated by $c_i[i_1,i_2,i_3]$ and
$c_ic_j[i_1,i_2,i_3]$ where $i_1,i_2,i_3$ are not all the same.
\end{enumerate}
Use the $q$-grading (i.e., $|c_i| = 1$ and $|d_i| = 0$) and write $T_i\subseteq T$ for the
component in grading $i$.
Additively, Lemma \ref{lem:alternating} gives $T_0^{C_3} \isom S\{1, \delta \}$, and by observing the orbit types above of
degree 3 we see that $T_3^{C_3} = S\{ c_1c_2c_3, c_1c_2c_3\delta \}$. This covers orbit types
(1), (3), and (4). Orbits of type (2) have fixed points
$(c_1+c_2+c_3)[a,a,a] = \tr(c_1)s_3^a$ and $(c_1c_2 + c_2c_3 + c_3c_1)[a,a,a] =
\tr(c_1c_2)s_3^a$.
It remains to show that fixed points of orbit type (5) can be written as an
$S$-combination of the generators \eqref{eq:c-gens}.
We focus on the $c_i[i_1,i_2,i_3]$ case because the other case is similar.

Let $N \subseteq T^{C_3}$ denote the $S$-submodule generated by \eqref{eq:c-gens}.
The $C_3$-fixed points of the orbit $c_i[i_1,i_2,i_3]$ is generated over $\k$ by
$\tr(c_i[i_1,i_2,i_3])$.
We will show that $\tr(c_i[i_1,i_2,i_3])\in N$ by strong induction on
total $d$-homogeneous degree $h:=i_1+i_2+i_3$. For the base case, it is clear that
the only $C_3$-invariant polynomials with $h=1$ are $\tr(c_id_1)$ for $i=1,2,3$.
Let $h \geq 2$, and assume that all terms of the form $\tr(c_i[i_1,i_2,i_3])$
with $i_1+i_2+i_3 < h$ are in $N$.

If $i_1,i_2,i_3$ are all $>0$, then $\tr(c_i[i_1,i_2,i_3]) = s_3\cdot
\tr(c_i[i_1-1,i_2-1,i_3-1])$, and the last term is in $N$ by the inductive
hypothesis. Thus we may assume that one of the indices is zero; without loss of
generality assume this is the last index. Now we will change
notation so that the new goal is to show that $\tr(c_i[a,b,0])\in N$ for some
$a,b$ with $a+b=h\geq 2$.

For fixed homogeneous degree $h$, we do induction on the quantity $D(a,b,0):=
|a-b|$. For the base case, assume $a=b$. The case $a=1$ is covered by the
generator $\tr(c_id_1d_2)$ of $N$, so we may assume $a \geq 2$. We have
\begin{align*}
(c_1d_1^{a-1}d_2^{a-1} & + c_2 d_2^{a-1}d_3^{a-1} + c_3 d_1^{a-1} d_3^{a-1})
\cdot (d_1d_2+d_2d_3+d_1d_3) 
\\ & = c_1 d_1^a d_2^a + c_2 d_2^a d_3^a + c_3 d_1^ad_3^a + \text{terms
divisible by $s_3$}.
\end{align*}
The first factor is in $N$ by the $h$-induction hypothesis, the second
factor is in $S$, and the terms divisible by $s_3$ can be written as $s_3$ times a term covered by the
$h$-induction hypothesis. This shows that $\tr(c_1d_1^ad_2^a)\in N$.

We also need another base case, $D = 1$. Without loss of generality
assume $a = b+1$. Since $h=a+b\geq 2$, we have $b\geq 1$.
If $b= 1$, then this is $\tr(c_id_1^2d_2)$. If $b\geq 2$, we have
\begin{align*}
\tr(c_id_1^2d_2)(d_1^{b-1}d_2^{b-1} + d_2^{b-1}d_3^{b-1}+d_1^{b-1}d_3^{b-1})
& = c_id_1^{b+1} d_2^b + c_{i+1}d_2^{b+1} d_3^b + c_{i+2}d_1^bd_3^{b+1}
\\ & \hspace{15pt} + \text{terms divisible by }s_3
\end{align*}
and the left hand side is in $N$. Thus $\tr(c_i d_1^{b+1}b_2^b)\in N$.

Now we do the inductive step in $D = |a-b|$ (still fixing $h\geq 2$) and assume $|a-b| \geq 2$.
Without loss of generality, $a> b$. 

\emph{Case 1: $b=1, a \geq 3$.} We have
\begin{align*}
(c_id_1^2d_2  & + c_{i+1}d_2^2 d_3 + c_{i+2}d_1d_3^2)(d_1^{a-2} + d_2^{a-2} + d_3^{a-2})
\\ & = c_id_1^a d_2 + c_{i+1}d_2^a d_3 + c_{i+2}d_1d_3^a + \tr(c_id_1^2d_2^{a-1})
+ \tr(c_id_1^2d_2d_3^{a-2}).
\end{align*}
The first factor on the left hand side is the generator $\tr(c_id_1^2d_2)$ of $N$, the
second factor is in $S$, and
we claim the last two terms are also in $N$: first,
$\tr(c_id_1^2d_2^{a-1})$ has $D(2,a-1) = a-3 < D(a,1)=a-1$ so this is in $N$ by
the $D$-induction hypothesis, and
$\tr(c_id_1^2d_2d_3^{a-2})$ is a multiple of $s_3$ so it can be written as $s_3$
times a term covered by the $h$-induction hypothesis. Thus $\tr(c_id_1^ad_2)\in
N$.

\emph{Case 2: $b=0,a\geq 2$.} We have
\begin{align*}
(c_id_1 + &c_{i+1}d_2 + c_{i+2}d_3) (d_1^{a-1}+d_2^{a-1}+d_3^{a-1}) 
\\& = 
c_id_1^a + c_{i+1}d_2^a + c_{i+2}d_3^a + \tr(c_id_1d_2^{a-1}) + \tr(c_id_1d_3^{a-1})
\end{align*}
where the first factor on the left hand side is the generator $\tr(c_id_1)$ of
$N$ and the
second factor is in $S$. If $a=2,3$ then the last two terms on
the right are generators (note that $\tr(c_id_1d_3) = \tr(c_{i+1}d_1d_2)$). 
If $a \geq 4$, then the last two terms are in $N$ by Case 1 (using its analogous
version for $a=1$, $b\geq 3$). Thus $\tr(c_id_1^a)\in N$.

\emph{Case 3: $b>1, a>3$.} Let $s_{(i,j,k)} = \tr(d_1^id_2^jd_3^k) +
\tr(d_1^kd_2^jd_3^i)\in S$. Recall the abbreviated notation $[i,j,k] := d_1^id_2^jd_3^k$.
We have
\begin{align*}
\tr(c_id_1^2d_2) \cdot s_{(a-2,b-1,0)}  & = (c_i[2,1,0]  + c_{i+1}[0,2,1] + c_{i+2}[1,0,2])
\\ &\hspace{10pt} \cdot \big([a-2,b-1,0] + [0,a-2,b-1] + [b-1,0,a-2]
\\ &\hspace{18pt} + [b-1,a-2,0] + [0,b-1,a-2]+[a-2,0,b-1]\big)
\\ &= c_i[a,b,0] + c_i[b+1,a-1,0] + c_{i+1}[0,a,b] 
\\ & \hspace{15pt} + c_{i+1}[0,b+1,a-1] + c_{i+2}[b,0,a] + c_{i+2}[a-1,0,b+1]
\\ & \hspace{15pt} + \text{terms divisible by }s_3
\\ & = \tr(c_i[a,b,0])+ \tr(c_i[b+1,a-1,0])
+ \text{terms divisible by }s_3.
\end{align*}
The left hand side is in $N$, and since $a-b \geq 2$ we have $D(b+1,a-1,0) = a -
b -2 < D(a,b,0)$ and so $\tr(c_i[b+1,a-1,0])$ is covered by the
$D$-inductive hypothesis. Using the same argument as above for the $s_3$-divisible
terms, we conclude that $\tr(c_i[a,b,0])\in N$.
\end{proof}

In the next lemma, we point out that some of the generators above are redundant.
\begin{lemma}\label{lem:fbar}
The elements $1,c_1c_2c_3,\delta,c_1c_2c_3\delta$, and
\begin{align*}
f_0 & = \tr(c_1)  & f_2 & = \tr(c_1d_2) & f_4 & = \tr(c_1d_2d_3)
\\f_1  & = \tr(c_1d_1)  & f_3 & = \tr(c_1d_1d_2) &  f_5 & = \tr(c_1d_1^2d_2)
\\F_0 & = \tr(c_1c_2) & F_2 & = \tr(c_1c_2d_2) & F_4 & = \tr(c_1c_2d_2d_3)
\\F_1  & = \tr(c_1c_2d_1)  & F_3 & = \tr(c_1c_2d_1d_2) &  F_5 & = \tr(c_1c_2d_1^2d_2).
\end{align*}
form a set of $S$-module generators of $T^{C_3}$.
\end{lemma}
\begin{proof}
Note that $\tr(c_3d_1)= \tr(c_1d_2)$ and $\tr(c_3d_1d_2)=\tr(c_1d_2d_3)$.
If $\{ X_1,X_2,X_3 \}$ is an orbit of monomials in $S$, we have
$$ (c_1+c_2+c_3)(X_1+X_2+X_3) = \tr(c_1X_1) + \tr(c_2X_1) + \tr(c_3X_1), $$
and so $\tr(c_2d_1)$, $\tr(c_2d_1d_2)$, $\tr(c_2d_1^2d_2)$, and
$\tr(c_2d_1d_2^2)$ from \eqref{eq:c-gens} are not needed as $S$-module generators.
A similar observation eliminates one generator (say, the $(i,j)=(2,3)$ indexed one) from
every family in the third row of \eqref{eq:c-gens}. Moreover, we have
\begin{align*}
\tr(c_3d_1^2d_2)  & = f_5 + f_2s_2 - f_3s_1
\\\tr(c_1d_1d_2^2) & = -f_5 - f_0s_3 + f_3s_1
\\\tr(c_3d_1d_2^2) & = -f_5 - f_0s_3 - f_2s_2 + f_3s_1 + f_4s_1
\\\tr(c_3c_1d_1^2d_2)  & = F_5 + F_2s_2 - F_3s_1
\\\tr(c_1c_2d_1d_2^2) & = -F_5 - F_0s_3 + F_3s_1
\\\tr(c_3c_1d_1d_2^2) & = -F_5 - F_0s_3 - F_2s_2 + F_3s_1 + F_4s_1.
\end{align*}
The remaining generators from \eqref{eq:c-gens} are the ones listed in the lemma
statement.
\end{proof}

\begin{proposition}\label{prop:c1-d1}
As an $S$-module, $T^{C_3} = \big(\k[d_1,d_2,d_3]\tensor \Lambda_{\k}[c_1,c_2,c_3]\big)^{C_3}$ is
free on the generators $\{ 1, \bar{c},\delta, \bar{c}\delta,f_0,\dots,f_5,F_0,\dots,F_5 \}$
where $\bar{c}=c_1c_2c_3$ and the generators $f_i,F_i$ are as in Lemma
\ref{lem:fbar}.
\end{proposition}
The proof, which involves a calculation using {\tt sage}, is given in the appendix. We summarize the results, using the grading $|x_i| = -2$ from Conjecture \ref{conj:action}.

\begin{corollary}\label{cor:E20*}
Let $S_\k = \k[s_1,s_2,s_3]$ with generators defined in Lemma \ref{eq:E20}.
In the Serre spectral sequence \eqref{eq:E20}, 
$\SE_2^{0,*}$ is a free module over $S_\k[b_1^\pm]$ with generators
$1,\bar{c},\delta,  \bar{c}\delta, f_0,\dots,f_5, F_0,\dots,F_5$. The degrees are given as follows, written as $(p,q,t)$ where $t$ is internal degree.
\begin{align*}
|b_1| & = (0,2,0) & |\bar{c}| & = (0,3,-6) & |\delta|  & = (0,0,-18) &
|\bar{c}\delta| & = (0,3,-24)
\\|f_0| & = (0,1,-2) & |f_1| & = (0,1,-8) & |f_2| & = (0,1,-8)
\\|f_3| & = (0,1,-14) & |f_4| & = (0,1,-14) & |f_5| & = (0,1,-20)
\\ |F_0| & = (0,2,-4) 
 & |F_1| & = (0,2,-10) & |F_2| & = (0,2,-10)
\\|F_3| & = (0,2,-16) & |F_4| & = (0,2,-16) & |F_3| & = (0,2,-22)
\end{align*}
\end{corollary}

\begin{proposition}\label{prop:serre-E2-summary}
There is an isomorphism of $S^*_\k := \k[b_1^\pm,b_2,s_1,s_2,s_3]\tensor \Lambda_\k[a_2]$-modules
\begin{align*}
\SE_2^{*,*} & \isom {S^*_\k\{
1,\bar{c},\delta,\bar{c}\delta,f_0,\dots,f_5,F_0,\dots,F_5 \}\over 
\big( a_2 s_1,b_2 s_1, a_2s_2, b_2 s_2, a_2 \delta, b_2 \delta,
a_2\bar{c}\delta,b_2\bar{c}\delta, a_2 f_i, b_2
f_i,a_2 F_i, b_2F_i \big)}.
\end{align*}
The generators are defined in Lemma \ref{eq:E20} and Lemma \ref{lem:fbar} with
$|b_2| = (2,0,0)$ and $|a_2| = (1,0,0)$, and
$\Delta = s_3$.
\end{proposition}
\begin{proof}
This comes from putting together Proposition \ref{prop:p>0} and Corollary \ref{cor:E20*}. Since $s_1,s_2,\delta,\bar{c}\delta, f_i$, and $F_i$
come from free summands and $(a_2,b_2)\subseteq H^*(C_9/C_3; \k) \isom \k[b_2]\tensor \Lambda_\k[a_2]$ acts trivially on $H^*(C_9/C_3; \k[C_3])$, these generators are killed by $a_2$ and $b_2$.
\end{proof}

\section{Handling the $I$-completion} \label{sec:completion}
Let $M = \Sym(\Ind_{C_3}^{C_9} \bar{\rho})$. 
In this section we calculate the localized Serre spectral sequence $E_2$ term 
$$ \SE_2^{*,*}((\Delta^{-1}M)^\hhat_I) = H^*(C_9/C_3; \hat{H}^*(C_3; (\Delta^{-1} M)^\hhat_I))$$
by relating it to $\SE_2^{*,*}(M) = H^*(C_9/C_3;
\hat{H}^*(C_3; M))$, which we have already computed in Proposition
\ref{prop:serre-E2-summary}. As in the non-completed version, this occurs in two stages: first we compute
$\hat{H}^*(C_3; (\Delta^{-1} M)^\hhat_I)$ in Lemma \ref{lem:completed-C3-coh} by
mimicking the strategy used in Section \ref{sec:p>0} with a basis that is more
suitable to working with $I$. Next, we attack the outer cohomology by
proving a seemingly general statement about cohomology of completions
(Lemma \ref{lem:I_ki}) which nevertheless requires explicit calculations
(Lemma \ref{lem:w-basis}) in order to construct a basis with the required
properties. In both parts, we
use a different basis from the ones used in Sections \ref{sec:p>0}
and \ref{sec:c1-d1}, but the work here does not replace those sections,
as we rely on them for explicit generators and relations for $\SE_2^{*,*}(M)$.
In particular, the calculations in Section \ref{sec:c1-d1} would not be
tractable using the bases chosen in this section.

The following lemma is used to prove Lemma \ref{lem:ML}.

\begin{lemma}[{\cite[Ch. 3, Cor. 1.1]{lubkin}}] \label{lem:lim1}
Let $A$ be a ring and $(C^*_k)_{k\geq 0}$ be an inverse system of chain
complexes of left $A$-modules. Let $C^* = \llim_k C^*_k$. Assume $C^n\to C^n_k$
are epimorphisms for all $k\geq 0$. Then there is a short exact sequence
$$ 0\to {\lim}^1 H^{n-1}(C^*_k)\to H^n(C^*)\to \lim_k H^n(C^*_k)\to 0 $$
for all $n$.
\end{lemma}

\begin{lemma}\label{lem:ML}
We have
$$\hat{H}^*(C_3; (\Delta^{-1} M)^\hhat_I) \isom \lim_k \hat{H}^*(C_3; \Delta^{-1}
M/I^k).$$
\end{lemma}
\begin{proof}
Applying Lemma \ref{lem:lim1} to the free $\W[C_9]$-resolution for
$M[\Delta^{-1}]/I^k$, we have a short exact sequence
$$ 0\to {\lim_k}^1 H^{n-1}(C_3; M[\Delta^{-1}]/I^k) \to H^n(C_3;
M[\Delta^{-1}]^\hhat_I) \to \lim_k H^n(C_3; M[\Delta^{-1}]/I^k) \to 0. $$
To show the $\lim^1$ term vanishes, we show it satisfies the Mittag-Leffler
condition; that is, for fixed $k$, the images
$$ H^*(C_3; M[\Delta^{-1}]/I^{k+\ell}) \to H^*(C_3;
M[\Delta^{-1}]/I^k) $$
stabilize. This is satisfied because
$$ H^*(C_3; M[\Delta^{-1}]/I^k)\isom H^*(C_3; (M/I^k)[\Delta^{-1}]) \isom
H^*(C_3; M/I^k)[\Delta^{-1}],$$
the quotient map takes $\Delta$-towers to $\Delta$-towers,
and $H^*(C_3; M/I^k)$ is finite in a fixed internal degree.
\end{proof}

We will need the following lemma in order to be able to use Lemma \ref{lem:C3-kunneth} in an analogous way to its use in Lemma \ref{lem:coh-rhobar}.
\begin{lemma}\label{lem:tower-comparison}
Suppose $I_i$ is an ideal of $A_i$ for $i = 1,\dots,n$. Write $A = A_1\tensor
\dots \tensor A_n$ and consider $I = I_1 + \dots + I_n$ (viewing $I_i$ as an
ideal of $A$ via the embedding $x \mapsto 1\tensor \dots \tensor x \tensor \dots
\tensor 1$). For any functor $F:\Ab\to \Ab$, 
$$ \lim_k F(A/I^k) \isom \lim_k F(A_1/I_1^k \tensor \dots \tensor A_n/I_n^k). $$
\end{lemma}
\begin{proof}
Let $J_k = I_1^k + \dots + I_n^k$, so that $A/J_k \isom A_1/I_1^k \tensor
\dots\tensor A_n/I_n^k$. We have
$$ I^{nk} \subseteq J_k \subseteq I^k $$
and so the systems given by $\{ J_k \}$ and $\{ I^k \}$ are equivalent.
This gives a map between the towers $\{ A/J_k \} \to \{ A/I^k \}$, and hence a
map of towers $\{ F(A/J_k) \}\to \{ F(A/I^k) \}$ that is an isomorphism in the
limit.
\end{proof}

Recall 
$$M \isom \W[x_0,\dots,x_8]/(x_0+x_3+x_6, x_1+x_4+x_7, x_2+x_5+x_8) $$
where $C_9$ permutes the generators, $I = (3,x_0-x_1,\dots,x_7-x_8)$ and $\Delta
= \prod_{i=0}^8x_i$.
Make the following change of coordinates:
\begin{align*}
x & = x_0
\\z_1  & = x_1-x_0
\\z_2 & = x_4-x_3
\\z_3 & = x_3-x_0
\\z_4 & = x_2-x_0
\\z_5 & = x_5-x_3.
\end{align*}
This implies
\begin{align*}
x_6 & = -x_0 - x_3 = -2x - z_3
\\x_7  & = -x_1-x_4 = - 2x-z_1-z_2 - z_3 
\\x_8  & = -x_2 - x_5 = - 2x-z_3 - z_4-z_5.
\end{align*}
In this basis we have $I = (3,z_1,\dots,z_5)$ and
$$ \Delta = x(x+z_1)(x+z_4)(x+z_3)(x+z_2+z_3)(x+z_3+z_5)(-2x-z_3)(-2x - z_1-z_2-z_3)(-2x-z_3-z_4-z_5). $$

\begin{lemma} \label{lem:completed-C3-coh-part1}
Let $M_1 = \W[z_1,z_2]$, $M_2 = \W[x, z_3]$, $M_3 = \W[z_4,z_5]$, $I_1 =
(3,z_1,z_2)\subseteq M_1$, $I_2 = (3,z_3)\subseteq M_2$, and $I_3 =
(3,z_4,z_5)\subseteq M_3$.
Then
$$ \hat{H}^*(C_3; (\Delta^{-1}M)^\hhat_I) \isom \lim_k \big(
\hat{H}^*(C_3; M_1/I_1^k) \tensor d_1^{-1}\hat{H}^*(C_3; M_2/I_2^k)\tensor
\hat{H}^*(C_3; M_3; I_3^k) \big) $$
where $d_1 = x_0x_3x_6$.
\end{lemma}

\begin{proof}
We have a $\W[C_3]$-module decomposition of the homological
degree-1 part of $M$ into three 2-dimensional indecomposable modules
$$ \W\{ x_0,\dots,x_5 \} \isom \W\{ z_1,z_2 \} \dsum \W\{ x,z_3 \} \dsum \W\{ z_4,z_5 \} $$
which gives an isomorphism of $C_3$-modules $M \isom M_1 \tensor M_2 \tensor M_3$.
Note also $I \isom I_1+I_2+I_3$.

Next we show that
$$ (\Delta^{-1}M)/I^k \isom (d_1^{-1} M)/I^k $$
for all $k$.
Observe that $\Delta$ can be written as a product of $d_1 = x(x+z_3)(-2x-z_3)$ with factors of
the form $x+z$ where $z\in I$.
To show that the other factors are invertible in the quotient, write
\begin{align*}
{1\over x+z} & = x^{-1}\cdot {1\over 1+ x^{-1} z} = x^{-1}\big(1- x^{-1} z + (x^{-1} z)^2
 - \dots \big) \in M^\hhat_I
 \\ & \equiv x^{-1}\big(1- x^{-1} z + (x^{-1} z)^2
 - \dots + (-1)^{k-1} (x^{-1} z)^{k-1}\big) \in M/I^k
\end{align*}
and hence $x+z$ is invertible in $M/I^k$.

We have
\begin{align*}
\hat{H}^*(C_3; (\Delta^{-1}M)^\hhat_I) 
& \isom \lim_k \hat{H}^*(C_3; \Delta^{-1}M/I^k)
\\ & \isom \lim_k \hat{H}^*(C_3; \Delta^{-1}(M_1\tensor M_2\tensor M_3)/I^k)
\\& \isom \lim_k \hat{H}^*(C_3; (M_1\tensor d_1^{-1}M_2\tensor M_3)/I^k)
\\ & \isom \lim_k \hat{H}^*(C_3; M_1/I_1^k \tensor d_1^{-1} M_2/I_2^k \tensor M_3/I_3^k)
\\ & \isom \lim_k \big(\hat{H}^*(C_3; M_1/I_1^k) \tensor \hat{H}^*(C_3;
d_1^{-1}M_2/I_2^k) \tensor \hat{H}^*(C_3; M_3/I_3^k)\big)
\end{align*}
where the first isomorphism is by Lemma \ref{lem:ML}, the fourth isomorphism
is by Lemma \ref{lem:tower-comparison}, and the last isomorphism is the K\"unneth isomorphism (Lemma \ref{lem:C3-kunneth}).
Finally, use the fact that localization commutes with cohomology.
\end{proof}


The calculation of the three tensor factors in Lemma
\ref{lem:completed-C3-coh-part1} are similar so we unify their proofs.
\begin{lemma}\label{lem:xy}
Suppose $\W\{x,y\}$ has $C_3$-action given by $\gamma(x) = y$, $\gamma(y) =
-x-y$ where $C_3 = \an{\gamma}$. Let $\delta = xy(x+y)$.
If $I = (3,x,y)$, then $$ \lim_k\hat{H}^*(C_3; \W[x,y]/I^k)\isom \hat{H}^*(C_3;
\W[x,y])^\hhat_\delta. $$
If $I = (3,y-x)$, then
$$ \lim_k \delta^{-1}\hat{H}^*(C_3; \W[x,y]/I^k) \isom \delta^{-1}\hat{H}^*(C_3; \W[x,y]). $$
\end{lemma}
\begin{proof}
An alternate proof of Lemma \ref{lem:coh-rhobar} would include the fact that (using the notation in this
case) there is a $\W[C_3]$-module decomposition of $\W[x,y]$ with
homogeneous summands, such that
the non-free summands of $\W[x,y]$ are the trivial summands $\W\{ \delta_1^i \}$
and the 2-dimensional indecomposable summands $\W\{ x\delta^i, y\delta^i \}$
whose cohomology is generated by $\delta^i$ times the exterior generator $\epsilon \in \hat{H}^1(C_3; \W\{x,y\})$.
Write $\W[x,y] = \dsums_i \W\{ \delta^i \}\dsum \dsums_i \W\{
x\delta^i,y\delta^i \}\dsum F$ where $F$ is the free component.
For each summand $S$ in the above decomposition, define
$$B_k(S) = \Im\big( \dots \to \hat{H}^*(C_3; S/(I^{k+2}\ints S)) \to \hat{H}^*(C_3;
S/(I^{k+1}\ints S))\to \hat{H}^*(C_3; S/(I^k\ints S))\big). $$
Using the fact that $I^k = \dsums_S (S\ints I^k)$,
we have that $\lim_k \dsums_S B_k(S) \isom \lim_k \dsums_S \hat{H}^*(C_3; S/(I^k\ints S)) \isom
\lim_k \hat{H}^*(C_3; \W[x,y]/I^k)$.

Suppose $F_i\isom \W[C_3]$ is a component of $F$ and $k \geq 0$. If $I = (3,x,y)$, then $F_i/(I^k\ints F_i)\isom \W/3^j[C_3]$ for some $j$. If $I = (3,y-x)$, then one can show that the chain $\dots\to F_i/(I^{k+1}\ints F_i\to F_i/(I^k \ints F_i))$ used to define $B_k(F_i)$ factors through $\W/3^j[C_3]$ for some $j$. In both cases, we can thus use the fact that $\hat{H}^*(C_3; \W/3^j[C_3])=0$. Thus we will ignore free summands in the calculation that follows.

It is also worth observing that if $S \isom \W\{ 1 \}$ (where $1\notin I$ has
trivial action), we have $\hat{H}^*(C_3; \W/3^k) \isom \hat{H}^*(C_3; \k)
\isom \k[a,b^\pm]/a^2$, but the odd-degree class $a$ is generated by $3^{k-1}\in
\W/3^k$ in the minimal resolution \eqref{eq:margolis}, and so $a\mapsto 0$ in
the chain defining $B_k(S)$. Thus
\begin{equation}\label{eq:W(1)} B_k(\W\{ 1 \}) \isom \k[b^\pm] \isom \hat{H}^*(C_3; \W).
\end{equation}

\emph{Case 1: $I = (3,x,y)$.}
We will show that
$$ \dsums_S B_{3n}(S) \isom \hat{H}^*(C_3; \W)[\delta,\epsilon]/(\delta^n,
\epsilon^2). $$
If $S = \W\{ \delta^i \}$, then 
$$ S/(I^k\ints S) =  \begin{cases} \W/3^{k-3i}\{ \delta^i \} & 3i< k\\0 & \text{otherwise} \end{cases}$$
and so
$$ \hat{H}^*(C_3; S/(I^k\ints S)) = \begin{cases} \hat{H}^*(C_3; \W/3^{k-3i})\{ \delta^i \} & 3i < k\\0 & \text{otherwise}.\end{cases} $$
By \eqref{eq:W(1)},
$$ B_k(S) \isom \begin{cases} \hat{H}^*(C_3; \W)\{ \delta^i \}  & 3i<k
\\ 0 & \text{otherwise}. \end{cases}$$
If $S = \W\{ x\delta^i, y\delta^i \}$, then
$$ S/(I^k\ints S) =  \begin{cases} \W/3^{k-(3i+1)}\{ x\delta^i, y\delta^i \} & 3i+1< k\\0 & \text{otherwise} \end{cases}$$
and so
$$ \hat{H}^*(C_3; S/(I^k\ints S)) = \begin{cases} \hat{H}^*(C_3; \W/3^{k-(3i+1)})\{ \epsilon\delta^i \}
& 3i+1 < k\\0 & \text{otherwise}\end{cases} $$
and similarly
$$ B_k(S) = \begin{cases} \hat{H}^*(C_3; \W)\{ \epsilon\delta^i \} & 3i+1<k\\
0 & \text{otherwise.}
\end{cases} $$
Taking the limit, we have $\lim_k B_k(S) \isom \hat{H}^*(C_3;
\W)[[\delta]][\epsilon]/(\epsilon^2) \isom (\hat{H}^*(C_3;
\W[x,y]))^\hhat_\delta$ using Lemma
\ref{lem:coh-rhobar}.

\emph{Case 2: $I = (3,y-x)$.} The main difference from Case 1 is that
$\delta\notin I$. We will show that
$$ \dsums_S B_k(S) \isom \hat{H}^*(C_3; \W)[\delta,\epsilon]/(\epsilon^2). $$
If $S = \W\{ \delta^i \}$, then 
$S/(I^k\ints S) \isom \W/3^k\{ \delta^i \}$ and so $B_k(S) = \hat{H}^*(C_3;
\W)\{ \delta^i \}$ similarly to Case 1.
If $S = \W\{ x\delta^i, y\delta^i \}$, then $S/(S\ints I^k) = \W/3^k\{
x\delta^i,y\delta^i \}/(3^{k-1}(y-x))$. To calculate $\hat{H}^*(C_3; S/(I^k\ints S))$ in this case, use the minimal
resolution from \eqref{eq:margolis}: $\gamma$ has matrix $\begin{bmatrix} 1 & -3\\1 & -2 \end{bmatrix}$ with respect to the basis $\{ x\delta^i, (y-x)\delta^i \}$, so we have $1+\gamma+\gamma^2 = 0$ and
$1-\gamma$ has matrix $\begin{bmatrix} 0 & -3\\1 & -3 \end{bmatrix}$. The resolution
$$ \dots\too{0} \W/3^k\{ x\delta^i \} \dsum \W/3^{k-1}\{ (y-x)\delta^i \}
\too{1-\gamma} \W/3^k\{ x\delta^i \} \dsum
\W/3^{k-1}\{ (y-x)\delta^i \} \too{0} \dots $$
has even cohomology $\ker(1-\gamma)$ generated by $(3^{k-1}x +
3^{k-2}(y-x))\delta^i$,
and odd cohomology $(\W/3^k\{ x\delta^i \}\dsum \W/3^{k-1}\{ (y-x)\delta^i \})/\Im(1-\gamma)$
generated by $x\delta^i$. Then $\delta^i\epsilon$ is represented by $x\delta^i$ and the other class does
not lift up the chain. Thus we have
$$ B_k(S)\isom \hat{H}^*(C_3; \W)\{ \delta^i\epsilon \} $$
and the result follows as in Case 1.
\end{proof}

\begin{lemma} \label{lem:completed-C3-coh}
We have
\begin{align*}
\hat{H}^*(C_3; (\Delta^{-1}M)^\hhat_I) 
& \isom \Big(\Delta^{-1}\k[b_1^\pm, d_1,d_2,d_3]\tensor
\Lambda_\k[c_1,c_2,c_3]\Big)^\hhat_{(d_1-d_2, d_1-d_3)}.
\end{align*}
\end{lemma}
\begin{proof}
Starting with the right hand side of Lemma \ref{lem:completed-C3-coh-part1}, apply the first case of Lemma \ref{lem:xy} to $M_1$ and $M_3$ and the second
case to $M_2$. Let $\delta_1 = z_1z_2(z_1+z_2)$, $\delta_2 = x(x+z_3)(2x+z_3)$,
and $\delta_3 = z_4z_5(z_4+z_5)$ be the classes
that appear in this lemma as applied to $M_1$, $M_2$, and $M_3$, respectively. Note
that $\delta_2 = x_0x_3(x_0+x_3) = -x_0x_3x_6 = -d_1$.
Combining Lemma \ref{lem:completed-C3-coh-part1}, Lemma \ref{lem:xy}, and the
K\"unneth isomorphism (Lemma \ref{lem:C3-kunneth}), we have
\begin{align*}
\hat{H}^*(C_3 (\Delta^{-1}M)^\hhat_I) & \isom \hat{H}^*(C_3;
\W[z_1,z_2])^\hhat_{\delta_1} \tensor d_1^{-1}\hat{H}^*(C_3;
\W[x,z_3])\tensor \hat{H}^*(C_3; \W[z_4,z_5])^\hhat_{\delta_3}
\\ & \isom \big(d_1^{-1}(\hat{H}^*(C_3; \W[z_1,z_2])\tensor \hat{H}^*(C_3;\W[x,z_3])\tensor
\hat{H}^*(C_3; \W[z_4,z_5]))\big)^\hhat_{(\delta_1,\delta_3)}
\\ & \isom (d_1^{-1}\hat{H}^*(C_3; M))^\hhat_{(\delta_1,\delta_3)}.
\end{align*}
Recall from Lemma \ref{lem:inner-coh} that $d_1 = x_0x_3(-x_0-x_3)$, $d_2 =
x_1x_4(-x_1-x_4)$, and $d_3 = x_2x_5(-x_2-x_5)$.
One shows that $\delta_1 = d_1-d_2$ and $\delta_3 = d_1-d_3$ in $\hat{H}^0(C_3; M)$
(though these are not equal in $H^0$). We have
$$ (d_1^{-1}\hat{H}^*(C_3;M))^\hhat_{(d_1-d_2, d_1-d_3)} \isom (\Delta^{-1}\hat{H}^*(C_3;M))^\hhat_{(d_1-d_3, d_1-d_3)}  $$
because $\Delta = d_1d_2d_3$ and $d_2 = d_1 + (d_2-d_1)$ and $d_3 = d_1 +
(d_3-d_1)$ are invertible after inverting $d_1$ and completing at
$(d_1-d_2,d_1-d_3)$ by a similar argument as in Lemma \ref{lem:completed-C3-coh-part1}. Finally, plug in the formula for $\hat{H}^*(C_3;
M)$ from Lemma \ref{lem:inner-coh}.
\end{proof}

\begin{lemma}\label{lem:MN}
Suppose $f:M \to N$ is a surjective map where $M$ and $N$ are indecomposable
$\k[C_3]$-modules. If $f$ is not an isomorphism, then it sends the fixed
points of $M$ to zero.
\end{lemma}
\begin{proof}
There are three possible indecomposable $\k[C_3]$ modules: the trivial module
$\k\{ a_1 \}$, the free module $\k\{ a_1, a_2, a_3\}$, and the
2-dimensional indecomposable module $\k\{ a_1,a_2 \}$ where $\gamma(a_1) = a_2$
and $\gamma(a_2) = -a_1-a_2$ (here $\gamma$ is the generator of $C_3$).

\emph{Case 1: $M = \k\{ a_1,a_2,a_3 \}$ and $N = \k\{ a_1 \}$ or $N =
\k\{ a_1,a_2 \}$.} Consider the induced map $f_*$ on group
cohomology. The generator of $H^0(C_3; M)$ is annihilated by the
$C_3$-cohomology periodicity operator $b\in H^2(C_3; \k)$, whereas the
generator of $H^0(C_3; N)$ is $b$-periodic for the modules $N$ under
consideration. Thus $f_*$ sends the generator of $H^0(C_3; M)$ to zero.

\emph{Case 2: $M = \k\{ a_1,a_2 \}$ and $N = \k\{ a_1 \}$.} The $\k$-fixed
points of $M$ are generated by $a_1-a_2$, and since $N$ consists of fixed
points, we have $f(a_2) = f(\gamma(a_1)) = \gamma(f(a_1)) = f(a_1)$, and so
$f(a_1-a_2)=0$.
\end{proof}

\begin{lemma} \label{lem:I_ki}
Suppose a $\k[C_3]$-algebra $A$ has a decomposition $\dsums_i A_i$ into
$\k[C_3]$-indecomposable factors $A_i$ and
$I\subseteq A$ is an ideal that is preserved by the
$C_3$-action. Assume $\intss_{k\geq 0} I^k = 0$. Moreover, assume that, for every $k\in \N$, $I^k$ decomposes as
$\k[C_3]$-modules into $\dsums_i I_{k,i}$ for $\k[C_3]$-submodules $I_{k,i} \subseteq A_i$.
Define a filtration 
$$ F^kH^*(C_3; A) = \dsums_\attop{i\text{ such that}\\(A_i)^{C_3}\subseteq I^k} H^*(C_3; A_i). $$
Then
$$ \lim_k H^*(C_3; A/I^k) \isom \lim_k H^*(C_3; A)/F^kH^*(C_3; A). $$
\end{lemma}
Note that $F^k H^0(C_3; A) = \dsums_i H^0(C_3; A_i)\ints I_{k,i} = H^0(C_3; A) \ints I^k$.
The proof strategy is based on the proof of Lemma \ref{lem:xy}. This version is
more complicated in that it uses $H^*$ instead of $\hat{H}^*$, but less complicated
in that the modules are over $\k$ instead of $\W$.
\begin{proof}
Define 
$$B_{k,i} = \Im\big( \dots \to H^*(C_3; A_i/I_{k+2,i}) \to H^*(C_3; A_i/I_{k+1,i})\to
H^*(C_3; A_i/I_{k,i})\big). $$
Then $\lim_k \dsums_i B_{k,i} \isom \lim_k \dsums_i H^*(C_3; A_i/I_{k,i}) \isom
\lim_k H^*(C_3; A/I^k)$.

By definition we have a decomposition
\begin{align*}
H^*(C_3; A)/F^kH^*(C_3; A) &  \isom \dsums_i H^*(C_3; A_i)\Big/
\dsums_\attop{i\text{ such that}\\(A_i)^{C_3}\subseteq I_{k,i}} H^*(C_3; A_i).
\end{align*}
Write $\pr_{k,i}$ for the projection to the $i^{th}$ summand of $H^*(C_3;A)/F^k H^*(C_3; A)$.
Note that $\pr_{k,i} H^*(C_3; A_i)$ is either 0 or $H^*(C_3; A_i)$.
It suffices to show that
\begin{equation}\label{eq:AB} B_{k,i}\isom \pr_{k,i} H^*(C_3; A_i). \end{equation}

Consider an indecomposable module $A_i$ and a power $k$.
If $I_{k,i} = 0$ then the two sides of \eqref{eq:AB} both equal $H^*(C_3; A_i)$:
for the left hand side, observe that $I_{k+\ell,i} \subseteq I_{k,i}$ for
$\ell\geq 0$ so all the
maps in the chain are the identity on $H^*(C_3; A_i)$.
If $A_i = I_{k,i}$ then the fixed point is in $I_{k,i}$ and both
sides of \eqref{eq:AB} are zero.

The remaining case is that $I_{k,i}$ is a proper nontrivial submodule of $A_i$.
Let
$k+m$ be the largest index where $I_{k+m,i}\neq 0$; this exists by finiteness of
$A_i$ and the condition about $\intss_{k\geq 0}I^k$.
Apply Lemma
\ref{lem:MN} to the map $A_i/I_{k+m+1,i} = A_i\to A_i/I_{k,i}$ to see that the
fixed point $\phi\in A_i$ is in $I_{k,i}$ and
the image of $H^0(C_3; A_i/I_{k+m+1,i})\to H^0(C_3; A_i/I_{k,i})$ is zero.
Thus $\pr_{k,i} H^*(C_3; A_i)= 0$ and $B_{k,i} = 0$ in the case of homological
degree $* = 0$. To see that we also have $B_{k,i}=0$ for other homological
degrees, observe that for the three types of $\k[C_3]$ indecomposable modules
(see the proof of Lemma \ref{lem:MN}), higher cohomology is either zero (for
free modules) or isomorphic to $H^0$.
\end{proof}

\begin{lemma}\label{lem:w-basis}
There is a $\k[C_3]$-module basis for
$$ \k[b_1^\pm, d_1, d_2, d_3]\tensor \Lambda_\k[c_1,c_2,c_3] $$
that satisfies the criteria in Lemma \ref{lem:I_ki} with respect to the ideal $J
= (d_1-d_2, d_1-d_3)$.
\end{lemma}
\begin{proof}
Since $b_1$ has trivial action and $I$-filtration zero we can ignore
it; the desired basis will be free over $b_1$ on a basis for
$\k[d_1,d_2,d_3] \tensor \Lambda_\k[c_1,c_2,c_3]$.
In order to construct the
decomposition in the hypotheses for Lemma
\ref{lem:I_ki}, we will use the following change of basis:
\begin{align*}
d & = d_1
\\w_1 & = d_2-d_1
\\w_2 & = d_1+d_2+d_3
\end{align*}
which satisfies $\gamma(d)= d+w_1$, $\gamma(w_1)= w_1+w_2$, and
$\gamma(w_2)=w_2$.
Then $J = (w_1,w_2)$ and $s_3= d_1d_2d_3 = d(d+w_1)(d-w_1+w_2)$ is a fixed
point. We have a basis
$$ \k\{ d^\epsilon s_3^i w_1^j w_2^k \}_\attop{\epsilon\in \{ 0,1,2 \}\\ i,j,k\geq 0} $$
as there is a bijection between this and the monomial basis $\{ d^i w_1^jw_2^k \}$.
Write
$$ A = \k\{ d^\epsilon w_1^j w_2^k \}_\attop{\epsilon\in \{ 0,1,2 \}\\ j,k\geq
0}. $$
It suffices to find a basis of the desired form for
\begin{equation}\label{eq:AAA} A \dsum A\{ c_1,c_2,c_3 \}\dsum A\{ c_1c_2,
c_2c_3, c_1c_3 \} \dsum A\{ c_1c_2c_3 \},
\end{equation}
as the full basis will be free over $s_3$ on this.
Each summand is a $\k[C_3]$-module.
We work one $c$-degree at a time, as the action preserves this
degree, and start with $A$.

It will be helpful to understand the $\k[C_3]$-module structure of the
submodule $W = \{ w_1^jw_2^k \}_{j,k\geq 0} \subseteq J$. Let $\wbar =
w_1(w_1+w_2)(w_1-w_2)$. The non-free indecomposable summands are the trivial module
$\k\{ \wbar^i \}$ in homogeneous degree $3i$ and the indecomposable 2-dimensional module $\k\{
\wbar^iw_1, \wbar^iw_2 \}$ in homogeneous degree $3i+1$. The remaining
degrees consist of free modules. Write $A_n$ and $W_n$ for the homogeneous
degree $n$ parts of $A$ and $W$, respectively, and $F_n$ for the free module
component of $W_n$. Note that $\dim W_n = n+1$.
We have 
$$A_n = W_n \dsum W_{n-1}\{ d \} \dsum W_{n-2}\{ d^2 \} $$
though this is not a $\k[C_3]$-module decomposition. It suffices to find a
decomposition $A_n \isom \dsums_i A_{n,i}$ into indecomposable $\k[C_3]$-modules
with the following property: every $A_{n,i}$ has a basis $\beta(f_i) = \{ f_i,
\gamma f_i - f_i, f_i + \gamma f_i + \gamma^2 f_i \}$  such that the collection $\beta_n=\unions_i
\beta(f_i)$ reduces to a basis of the associated graded
$$ (\dsums_i J^i/J^{i-1}) \ints A_n = (J^n/J^{n-1} \dsum J^{n-1}/J^{n-2} \dsum
J^{n-2}/J^{n-3})\ints A_n = W_n \dsum dW_{n-1} \dsum d^2W_{n-2}. $$

\emph{Case 1: $n = 3i+1$.} We have $W_n = \k\{ w_1\wbar^i, w_2\wbar^i \} \dsum F_n$, $W_{n-1}
= \k\{ \wbar^i \}\dsum F_{n-1}$, and $W_{n-2} = F_{n-2}$.
We have 
$\beta(d\wbar^i) = \{ d\wbar^i, w_1\wbar^i, w_2\wbar^i \}$.
Choose $\k[C_3]$-module generators $\{ f_1,\dots,f_i \}$ of $F_n$, generators
$\{ f'_1,\dots,f'_i \}$ of $F_{n-1}$, and generators $\{ f''_1,\dots,f''_i \}$
of $F_{n-2}$.
Then each $\beta(f_i)$ reduces to three linearly independent elements in
$(J^n/J^{n-1})\ints A_n \isom W_n$, each $\beta(df'_i)$ reduces to three
linearly independent elements in $(J^{n-1}/J^{n-2})\ints A_n \isom dW_{n-1}$,
and each $\beta(d^2f''_i)$ reduces to three linearly independent elements in
$(J^{n-2}/J^{n-3})\ints A_n = d^2 W_{n-2}$.
Thus $$ \beta_n = \beta(d\wbar^i) \union \beta(f_1)\union \dots \union \beta(f_i)
\union \beta(df'_1)\union \dots \union \beta(df'_i) \union \beta(d^2f''_1)\union
\dots \union \beta(d^2f''_i)$$
is a basis for $A_n$ that satisfies the desired property.

\emph{Case 2: $n = 3i+2$.} We have $W_n = F_n$, $W_{n-1} = \k \{ w_1\wbar^i, w_2
\wbar^i\}\dsum F_{n-1}$, and $W_{n-2} = \k\{ \wbar^i \}\dsum F_{n-2}$. We
calculate
$$ \beta(d^2\wbar^i) = \{ d^2\wbar^i, (-dw_1 + w_1^2)\wbar^i, (-dw_2 -w_1^2 +
w_1w_2 + w_2^2)\wbar^i \}. $$
In the associated graded, this has the same span as $\{ dw_1\wbar^i, dw_2\wbar^i
\}\in (J^{n-1}/J^{n-2})\ints A_n \isom dW_{n-1}$ and $d^2\wbar^i \in
(J^{n-2}/J^{n-3})\ints A_n \isom d^2W_{n-2}$. We complete the basis $\beta_n$
using free summands similarly to the previous case.

\emph{Case 3: $n = 3i$.} We have $W_n = \k\{ \wbar^i \}\dsum F_n$, $W_{n-1}
= F_{n-1}$, and $W_{n-2} = \k\{ w_1\wbar^{i-1}, w_2\wbar^{i-1} \}\dsum F_{n-2}$.
We calculate
$$ \beta(d^2(w_1+w_2)\wbar^{i-1}) = \{ d^2(w_1+w_2)\wbar^{i-1}, (d^2 w_2 - dw_1^2+dw_1w_2 +
w_1^3 - w_1^2w_2)\wbar^{i-1}, (-w_1^3+w_1w_2^2)\wbar^{i-1} \}. $$
In the associated graded, this has the same span as $\{d^2w_1\wbar^{i-1},
d^2w_2\wbar^{i-1}\}$ in $(J^{n-2}/J^{n-3})\ints A_n \isom d^2 W_{n-2}$ and
$\wbar^i = (w_1^3-w_1w_2^2)\wbar^i \in J^n/J^{n-1}\ints A_n \isom W_n$. We complete the basis $\beta_n$
using free summands similarly to the other cases.

This concludes the construction of a basis for $A$. 
We now turn to the other
summands in \eqref{eq:AAA}. Since $c_1c_2c_3$ is fixed
by $C_3$, we immediately get a basis for $A\{ c_1c_2c_3 \}$. It suffices to construct a basis for $A\{
c_1,c_2,c_3 \}$ as the third term is analogous. We need to construct a
decomposition of $A_n\{ c_1,c_2,c_3 \}$ as indecomposable $\k[C_3]$-modules whose
$\k$-basis reduces to a basis of the associated graded
$$ W_n\{ c_1,c_2,c_3 \} \dsum dW_{n-1}\{ c_1,c_2,c_3 \} \dsum d^2 W_{n-2}\{
c_1,c_2,c_3 \}. $$
Suppose $\beta(f)$ is one of the bases from the construction above of a basis
for $A$. Since $c_1$, $c_2$, and $c_3$ are not in $J$, 
the basis $c_1\beta(f) \union
c_2 \beta(f) \union c_3 \beta(f)$ induces a basis of the associated graded with
the desired properties.
However, it need not come from a $\k[C_3]$-module decomposition. To address this, first we
claim we may replace one of the terms $c_i \beta(f)$ with $c_i\beta(\gamma
f)$. Clearly, $\beta(\gamma f)$ is a basis for the same $\k[C_3]$-module as
$\beta(f)$. It also has the same span on the level of the associated graded
because an element $g$ is in $J^k$ if and only if $\gamma g$ is in $J^k$.
It now suffices to show that there is a $\k[C_3]$-module decomposition of
$A_n\{ c_1,c_2,c_3 \}$ with basis $c_1\beta(f) \union c_2 \beta(\gamma f) \union c_3 \beta(\gamma^2
f)$.

To this end, consider the following decomposition
of $\k\{ c_1,c_2,c_3 \}\tensor \Span\beta(f)$:
\begin{align*}
\k\{ c_1f,\ c_2(\gamma f),\ c_3(\gamma^2 f) \}
\dsum \k\{ c_1(\gamma f - f),\ c_2 (\gamma^2 f - \gamma f),\ c_3 (f - \gamma^2 f) \}
\\\dsum\ \k\{ c_1(f + \gamma f + \gamma^2 f),\ c_2(f + \gamma f + \gamma^2 f),\ c_3(f +
\gamma f + \gamma^2 f) \}.
\end{align*}
The basis given can be grouped into $c_1\beta(f) \union c_2 \beta(\gamma f) \union c_3 \beta(\gamma^2
f)$.
\end{proof}

\begin{proposition}\label{prop:s1-s2-completion}
Write $N = \hat{H}^*(C_3; M)$. Then
$$ H^*(C_9/C_3; \hat{H}^*(C_3; (\Delta^{-1}M)^\hhat_I))\isom H^*(C_9/C_3; (\Delta^{-1} N)^\hhat_J) \isom H^*(C_9/C_3; \Delta^{-1} N)^\hhat_{(s_1,s_2)} $$
where $J = (d_1-d_2, d_1-d_3)$.
\end{proposition}
\begin{proof}
The first isomorphism is Lemma \ref{lem:completed-C3-coh}.
The second isomorphism will come from Lemma \ref{lem:I_ki} using the basis in Lemma \ref{lem:w-basis}.
Let $F$ be the filtration defined in Lemma \ref{lem:I_ki}. We start by
determining the filtration of the multiplicative generators of $H^*(C_9/C_3;
N)$ from Proposition \ref{prop:serre-E2-summary}. We have $J =
(w_1,w_2)$ in the notation of Lemma \ref{lem:w-basis}.
\begin{align*}
s_1 & = d_1+d_2+d_3 = w_2 \in J \backslash J^2
\\s_2 & = d_1d_2 + d_2d_3 + d_3d_1 = x(x+w_1) + (x+w_1)(x-w_1+w_2) +
(x-w_1+w_2)x 
\\ &\hspace{20pt} = - xw_2 -w_1^2 +w_1w_2 \in J \backslash J^2
\\s_3 & = \Delta = d_1d_2d_3 = x(x+w_1)(x-w_1+w_2)\notin J
\\b_1 & \notin J
\end{align*}
Instead of calculating the degrees for the other generators, we assert that they
do not matter after taking the limit:
$\Delta^{-1}N$ is finitely generated as a module over $\k[b_1^\pm,s_1,s_2,s_3^\pm]$, and so
there is a maximum $F$-degree $m$ of the module generators. Let $\til{J} = (s_1,s_2)$.
Then
$$ \til{J}^k \subseteq F^k H^0(C_9/C_3; \Delta^{-1} N) \subseteq \til{J}^{k-m} $$
and so completing with respect to the filtration $F$ is equivalent to completing
with respect to powers of $\til{J}$.
\end{proof}

\section{Higher Serre differentials and extensions}\label{sec:higher-serre}

Let $M = \Sym(\Ind_{C_3}^{C_9}\bar{\rho})$ be as in Section
\ref{sec:completion}.
We consider the Serre spectral sequence for Tate cohomology (Proposition \ref{lem:serre-tate})
\begin{equation}\label{eq:serre-sseq} E_2^{p,q} \isom H^{p}(C_9/C_3; \hat{H}^{q}(C_3; (\Delta^{-1} M)^\hhat_I))
\implies\hat{H}^{p+q}(C_9; (\Delta^{-1} M)^\hhat_I) \end{equation}
which has differentials
$$ d_r:E_r^{p,q} \to E_r^{p+r, q-r+1}. $$
The differentials preserve internal degree on homogeneous elements.
Combining Propositions \ref{prop:serre-E2-summary} and
\ref{prop:s1-s2-completion}, we have an expression for the Serre $E_2$ term:
\begin{equation}\label{eq:E2-sec5} \SE_2^{*,*} \isom  {S^*_\k\{
1,\bar{c},\delta,\bar{c}\delta,f_0,\dots,f_5,F_0,\dots,F_5 \}\over 
\big( a_2 s_1,b_2 s_1, a_2s_2, b_2 s_2, a_2 \delta, b_2 \delta,
a_2\bar{c}\delta,b_2\bar{c}\delta, a_2 f_i, b_2
f_i,a_2 F_i, b_2F_i \big)} \end{equation}
where $S^*_\k = \k[b_1^\pm, b_2, s_3^\pm][[s_1, s_2]]\tensor \Lambda_\k[a_2]$.

\begin{lemma} \label{lem:d3}
The only nontrivial differentials in the Serre spectral sequence 
\eqref{eq:serre-sseq} are generated over $\k[b_1^\pm,
b_2]\tensor \Lambda_\k [a_2]$ by
$$ d_3({\bar{c}}^\epsilon s_3^i\cdot a_2) = b_1^{-1} b_2^2 \bar{c}^\epsilon s_3^i $$
for $\epsilon \in \{ 0,1 \}$, $i\in \Z$.
There is a hidden extension $3\cdot b_1 = b_2$.
\end{lemma}
\begin{figure}[H]
\centerline{\includegraphics[width=400pt]{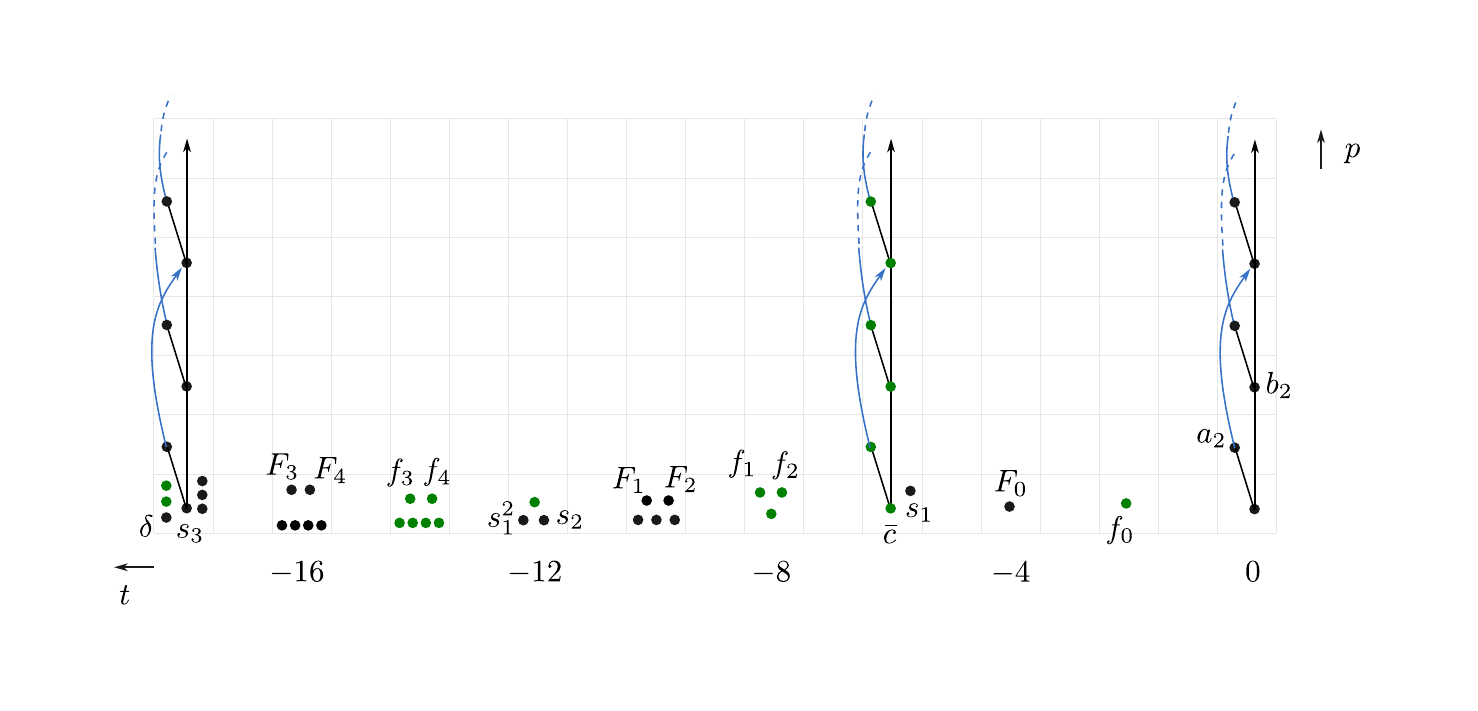}}
\caption{The Serre $E_3$ page with $E_3$ differentials. Each dot represents a copy of $\k[b_1^\pm]$. Green dots have odd $q$ degree. Differentials are in blue.}
\end{figure}
\begin{proof}
The target of a nontrivial differential must have $p > 0$, and since the only
multiplicative generators with $p>0$ are $a_2$ and $b_2$, the target must be in
$(\k[b_2]\tensor \Lambda_\k[a_2])\{ x \}$ where $x$ is not $b_2$-torsion.
From Proposition \ref{prop:serre-E2-summary}, we have $x\in
\k[b_1^\pm,s_3^\pm]\tensor \Lambda_\k[\bar{c}]$. The Leibniz rule prevents
nontrivial differentials with $b_2$-torsion source form having $b_2$-periodic
target, and so both the
source and target of the differential must be an $a_2$-$b_2$-tower. Since
differentials preserve the internal degree $t$ and the monomials
$\bar{c}^\epsilon s_3^i$ for $\epsilon \in \{ 0,1 \}$, $i\in \Z$ appear in distinct degrees, a differential takes a summand
$(\k[b_1^\pm, b_2]\tensor \Lambda_\k[a_2])\{ \bar{c}^\epsilon s_3^i \}$ to itself.
Thus we may focus on differentials within $\k[b_1^\pm, b_2]\tensor
\Lambda_\k[a_2]$, which is the image of 
the trivial coefficients spectral sequence in Section
\ref{sec:serre-triv}. By Lemma \ref{lem:serre-triv}, 
we find that $b_1$ and $b_2$ are permanent cycles, and
there is a $d_3$ differential $d_3(b_1a_2) = b_2^2$.
The extension also comes from comparison using Lemma \ref{lem:serre-triv}.
\end{proof}

\begin{theorem} \label{thm:main-completion}
Let $\til{S}^*_\k = (\W/9)[b_1^\pm,s_3^\pm][[s_1,s_2]]$. Then there is an
isomorphism of $\til{S}_\k^*$-modules
$$ \hat{H}^*(C_9; (\Delta^{-1}M)^\hhat_I) \isom \til{S}_\k^*\{ 1,\bar{c}, \delta, \bar{c}\delta,
f_0,\dots,f_5,F_0,\dots,F_5 \}\big/ \big(3s_1,3s_2,
3\delta,3\bar{c}\delta,3f_i,3F_i \big). $$
\end{theorem}
\begin{proof}
The spectral sequence converges by Lemma \ref{lem:serre-triv}; the boundedness
condition is satisfied by the fact that the $E_4$ page is concentrated in
degrees with $p=0,1,2,3$ since the differentials in Lemma \ref{eq:serre-sseq}
truncate all the $b_2$-towers.
For degree reasons, no other hidden 3-multiples are possible beyond the ones
identified in Lemma \ref{eq:serre-sseq}, and there are no
other relations involving elements of $S_\k$ that could be hiding hidden
extensions. Applying the differentials and extensions in Lemma
\ref{eq:serre-sseq} to the $E_2$ term in \eqref{eq:E2-sec5} gives the result.
\end{proof}
Observe that $\til{S}^*_\k/3 = \k[b_1^\pm,s_3^\pm][[s_1,s_2]]$. The degrees of the generators can be read off Corollary \ref{cor:E20*}, as a class in degree $(p,q,t)$ converges to a class in homological degree $p+q$ and internal degree $t$.

\begin{remark}
We opted not to write down the full ring structure of
$\hat{H}^*(C_9;(\Delta^{-1} M)^\hhat_I)$ or the Serre $E_2$ page, as the large number of relations
required make the presentation rather unwieldy. None of
the $S_\k$-module generators can be removed, and several families of relations
would need to be added, such as those coming from the general relation
$$ (c_1\phi_1 + c_2\phi_2 + c_3\phi_3)(c_1 \psi_1 + c_2\psi_2 + c_3\psi_3) =
\tr(c_1c_2 \phi_1 \psi_2) - \tr(c_1c_2 \phi_2 \psi_1) $$
where $\{ \phi_1,\phi_2,\phi_3 \}$ and $\{ \psi_1,\psi_2,\psi_3 \}$ are $C_9/C_3$-orbits in
$S_\k$. Several of these relations are given in the appendix (see Lemma
\ref{lem:extra-relations}). We also point out
the fact that $f_0 \cdot F_0 = \tr(c_1)\tr(c_1c_2) = 0$ in the Serre spectral
sequence, but we have $$ f_0 \cdot F_0 = 3 \bar{c} $$
in $\hat{H}^*(C_9; \Sym(\Ind_{C_3}^{C_9}\bar{\rho}))$.
\end{remark}

Finally we write down an $RO(C_9)$-graded version of Theorem
\ref{thm:main-completion}. This is a fairly simple modification of the
integer-graded statement, but we expect that the representation-graded
perspective will be useful when determining higher differentials; see e.g.
\cite{balderrama-total}.
Recall
$$ RO(C_9) \isom \Z\{ 1,\lambda_1,\lambda_2,\lambda_3,\lambda_4 \} $$
where $\lambda_i$ denotes the 2-dimensional real representation given by
rotation by $2i\pi \over 9$.
\begin{lemma}
Let $E$ be a $C_k$-ring spectrum and write $E^h = F(E{C_k}_+, E)$. There is a hfpss
$$ E_2^{n,V}=H^n(C_k; \pi_V(E^h))\implies \pi_{V-n}(E^h) $$
where $V$ represents a virtual representation in $RO(C_k) \isom
\Z[\lambda_1,\dots,\lambda_{(k-1)/2}]$. Moreover, 
$$ E_2^{*,\star} \isom H^*(C_k; \pi_*(E^h))[u_{\lambda_1}^\pm,\dots,u_{\lambda_{(k-1)/2}}^\pm] $$
where $u_{\lambda_i}\in H^0(C_k; \pi_2(S^{\lambda_i}\tensor E^h))$.
\end{lemma}
\begin{proof}
The first part of the proof works for an arbitrary finite group $G$ and
$V\in RO(G)$.
First observe that, since $S^V$ is dualizable, we have
\begin{align*}
E^h\tensor S^V &\hteq F(EG_+, E)\tensor S^V \hteq F(S, F(EG_+, E)\tensor S^V)
\hteq F(DS^V, F(EG_+, E))
\\ & \hteq F(DS^V \tensor EG_+, E) \hteq F(EG_+, S^V \tensor E)
\hteq (S^V\tensor E)^h
\end{align*}
and so there is a hfpss $E_2^{n,t}(V)\isom H^n(G; \pi_t(S^V\tensor E)) \implies
\pi_{n-t}(S^V\tensor E^h)$. To identify the $E_2$ term,
consider the Atiyah--Hirzebruch spectral sequence
$$ H_p(S^V; \pi_q E^h) \implies \pi_{p+q}(S^V \tensor E^h) $$
which collapses because it is concentrated in one $p$ degree. Moreover, we have
$$ H_*(S^V; \pi_* E^h) \isom H_*(S^V; \Z)\tensor \pi_* E^h. $$

Now specializing to $G= C_k$ and a basis element $V=\lambda_i\in RO(C_k)$, observe
that $H_*(S^{\lambda_i})\isom \Z[2]$ (where $[-]$ denotes degree shift) has
trivial $C_k$-action because rotation $S^V \to S^V$ is non-equivariantly
homotopic to the identity, and therefore the induced map on ordinary homology
has degree 1.
Thus for a fixed $V$ we have a spectral sequence
$$ E_2^{n,t}(V)=H^n(G; \pi_{t}E^h\tensor H_{|V|}(S^{V})) \isom H^{n}(G; \pi_t
E^h)[|V|] \implies \pi_{|V|+t-n}(S^V\tensor E^h). $$
The $E_2$ term $E_2^{*,*}(V)$ is a free rank-one module over the integer-graded $E_2$ term
$H^*(G; \pi_*E^h)$ with module generator $u_V = 1\tensor \iota_V \in H^0(G;
\pi_0E^h \tensor H_{|V|}(S^V))\isom H^0(G; \pi_{|V|}(S^V\tensor E^h))$ where $\iota_V$ is the generator of
$H_{|V|}(S^V)$. Thus we write $E_2^{n,t}(V) \isom H^n(G; \pi_* E^h)\{ u_V \}$.
Taking the sum over $V$ gives the $E_2$ term $E_2^{*,\star}$ for $\star\in RO(G)$.
\end{proof}

\begin{corollary}
There is an $RO(C_9)$-graded hfpss with $E_2$ term
$$ \hat{H}^{*,\star}(C_9; (\Delta^{-1}M)^\hhat_I) \isom \til{S}_\k^*\{ 1,\bar{c}, \delta, \bar{c}\delta,
f_0,\dots,f_5,F_0,\dots,F_5 \}\big/ \big(3s_1,3s_2,
3\delta,3\bar{c}\delta,3f_i,3F_i \big) [u_{\lambda_1}^\pm,\dots,u_{\lambda_4}^\pm]$$
with $u_{\lambda_i}$ in degree $(n,V) = (0,2-\lambda_i)$ and other notation as in Theorem
\ref{thm:main-completion}.
\end{corollary}

\section{Detection theorem}\label{sec:detection}

In this section we prove the $E_6$ detection theorem (Conjecture
\ref{conj:HHR-odd}(1)) using the work in \cite{HHR-odd}, which provides $p=3$
analogues of the $p=2$ arguments in \cite[\S11]{HHR-v2}.

Let $A = \Z_3[\zeta]$ where $\zeta$ is a primitive $9^{th}$ root of unity, and write $\pi = 1-\zeta$ for the uniformizer. Hill, Hopkins, and Ravenel \cite[\S4]{HHR-odd} construct a formal $A$-module
$F$ over $R_* := A[w^\pm]$ with logarithm
\begin{equation}\label{eq:log} \log_F(x) = x + \sum_{k>0}{w^{3^k-1}x^{3^k}\over \pi^k}. \end{equation}
The grading on $R_*$ is defined by $|w| = 4$. 

\begin{lemma}\label{lem:height}
$F$ as a formal group law over $R_*$ has height $\leq 6$.
\end{lemma}
\begin{proof}
To calculate the height of the graded formal group law $F$, it suffices to work
with the ungraded version; i.e., we can set $w=1$. Moreover, to prove the
desired inequality it suffices to prove that it has height equal to 6 after quotienting $R_*$ by its
maximal ideal $(\pi)$, which contains 3.

Let $\ell_k = \pi^{-k}$ be the coefficient of $x^{3^k}$ in the logarithm \eqref{eq:log}. By \cite[A2.2.1]{green} we have
$$ p \ell_n = \sum_{0\leq i< n} \ell_i v_{n-i}^{p^i} $$
where $v_n$ are the Hazewinkel generators. We can use this to solve for $v_i$
with $1\leq i\leq 6$. With coefficients taken modulo 3, we have
\begin{align*}
v_1 & = 2\zeta^5 + 2\zeta^4 + 2\zeta^3 + \zeta^2 + \zeta + 1
\\v_2 & = 2\zeta^4 + \zeta^3 + \zeta + 2
\\v_3 & = 2\zeta^3 + 1
\\v_4 & = \zeta^5 + \zeta^4 + \zeta^3 + \zeta^2 + \zeta + 1
\\v_5 & = \zeta^4 + 2\zeta^3 + \zeta + 2
\\v_6 & = 2\zeta^5 + \zeta^4 + \zeta^3 + 2\zeta^2 + 2.
\end{align*}
To check the image modulo the maximal ideal, set $\pi = 1 - \zeta= 0$; we
observe that $v_i \equiv 0$ modulo $1-\zeta$ for $i\leq 5$ and $v_6$ is a unit
modulo $1-\zeta$.
\end{proof}

\begin{proposition}\label{prop:detection}
Let $\Phi:\Ext_{BP_*BP}^{*,*}(BP_*, BP_*)\to \Ext_A^{*,*}(\F_p,\F_p)$ denote the Thom
reduction map.
Every $x\in \Ext_{BP_*BP}^{2,4\cdot 3^{j+1}}(BP_*, BP_*)$ such that $\Phi(x) =
b_j$ has nontrivial image in $H^2(C_9; \pi_{4\cdot 3^{j+1}}E_6)$.
\end{proposition}




\begin{proof}
Let $\bar{R}_* = R_*/3$.
It suffices to construct a composite
$$ \lambda: \Ext_{BP_*BP}^{*,*}(BP_*, BP_*)\to H^*(C_9; \pi_*E_6)\to H^*(C_9;
\bar{R}_*)$$
such that the image of $x$ is nonzero for every $x$ in the statement.
The first map is the map of Adams-Novikov $E_2$ terms induced by $BP_*\to \pi_*E_6$. To construct the second map, first note that the classifying map $BP_* \to R_*$ of the formal group law $F$ factors through $\pi_*E_6$ since it has height $\leq 6$ by Lemma \ref{lem:height}.
As a formal $A$-module, $F$ has an endomorphism $[\zeta]$ of order 9, with respect to which
the classifying map $[F]:\pi_*E_6 \to R_*$ is $C_9$-equivariant.
Thus, we obtain a map $\lambda':H^*(C_9; \pi_*E_6)\to H^*(C_9; R_*)$. 
We construct the second map in the composite $\lambda$ by composing $\lambda'$ with the quotient
$R_* \to R_*/3 =:\bar{R}_*$.
The
$C_9$-action on $R_*$ is given by $\gamma(w)=\zeta w$ for $C_9 = \an{\gamma}$, so the
action on $R_{4m}$ is multiplication by $\zeta^m$.

The 2-line $\Ext^{2,*}_{BP_*BP}(BP_*, BP_*)$ is additively generated by elements
$\beta_{i/j,k}$.
In particular, by \cite[\S4]{HHR-odd} the
Adams--Novikov $E_2$ term in the degree of $\beta_{3^j/3^j}$ is additively
generated by the set $\mathcal{B}_j \union \{ \beta_{3^j/3^j} \}$ where
$$ \mathcal{B}_j = \{ \beta_{c(j,k)/3^{j-2k}} \st 0 < k \leq j/2 \} $$
where $c(j,k) = (3^{j+1}+3^{j-2k})/4$. Since $\Phi(\beta_{3^j/3^j}) = -b_j$
\cite[Theorem 9.4]{MRW}, every $x$ in the statement can be written as a sum $\pm
\beta_{3^j/3^j}+y$ for $y\in \Span(\mathcal{B}_j)$.
Thus it suffices
to show that $\lambda(\beta_{3^j/3^j})\neq 0$ and $\lambda(\beta)=0$ for $\beta\in
\mathcal{B}$.

For the first part, we follow the proof of \cite[Lemma 11.6]{HHR-v2}.
By \cite{MRW} we have 
$$ \beta_{3^j/3^j} = \delta(\til{\beta}_{3^j/3^j})\text{ for }
\til{\beta}_{3^j/3^j} = \delta_{v_1^{3^j}}(v_2^{3^j} - v_1^{3^{j-2}\cdot
8}v_2^{3^{j-2}\cdot 7}) $$
where $\delta_{v_1^{3^j}}$ is the boundary map $\Ext^0(BP_*/(3,v_1^{3^j})) \to
\Ext^1(BP_*/3)$ and $\delta$ is the boundary map $\Ext^1(BP_*/3)\to \Ext^2(BP_*)$.
We calculate $\til{\beta}_{3^j/3^j}$ explicitly by computing the cobar
differential $d_1 = \eta_R - \eta_L = \eta_R - \mathrm{Id}$ where $\eta_R(v_1) \equiv v_1 \pmod 3$ and $\eta_R(v_2)\equiv v_2+v_1t_1^3-v_1^3t_1\pmod 3$:
\begin{align*}
d_1(v_2^{3^j} - v_1^{3^{j-2}\cdot 8}v_2^{3^{j-2}\cdot 7}) &\equiv
v_1^{3^j}t_1^{3^{j+1}} \pmod {(3,v_1^{3^j+6})}.
\\\til{\beta}_{3^j/3^j} = {1\over v_1^{3^j}}d_1(v_2^{3^j}-v_1^{3^{j-2}\cdot 8}v_2^{3^{j-2}\cdot 7}) &
\equiv t_1^{3^{j+1}} \pmod {(3, v_1^6)}
\end{align*}
for $j\geq 2$.

We have a diagram of long exact sequences
$$ \xymatrix{
\Ext^{1,4\cdot 3^{j+1}}(BP_*)\ar[r]\ar[d]^-\lambda & \Ext^{1,4\cdot
3^{j+1}}(BP_*/3)\ar[r]^-\delta\ar[d]^-\lambda & \Ext^{2,4\cdot 3^{j+1}}(BP_*)\ar[d]^-\lambda
\\H^1(C_9; R_{4\cdot 3^{j+1}})\ar[r] & H^1(C_9; R_{4\cdot 3^{j+1}}/3)\ar[r]^-{\bar{\delta}} &
H^2(C_9; R_{4\cdot 3^{j+1}})
}$$
where 
$$ H^1(C_9; R_{4\cdot 3^{j+1}})=\ker(1+\zeta^{3^{j+1}}+\dots + \zeta^{8\cdot
3^{j+1}})/\im(1-\zeta^{3^{j+1}}) = \ker(8)/0 = 0$$
for $j \geq 2$ (using the minimal resolution in Lemma \ref{lem:margolis}).
Thus $\bar{\delta}$ is an injection.
Analogously to \cite[(11.5)]{HHR-v2}, we have that $\lambda(t_1)$ is a unit.
Recall that $\lambda(v_1^6)=0$ in $R_*/3$, and so
$\lambda(\til{\beta}_{3^j/3^j}) = \lambda(t_1^{3^{j+1}})$ is a unit. Since
$\bar{\delta}$ is an injection and the square commutes, $\lambda(\delta(\til{\beta}_{3^j/3^j})) \neq 0$.

It remains to show that $\lambda(\beta)=0$ for $\beta\in \mathcal{B}_j$.
Hill, Hopkins, and Ravenel \cite{HHR-odd} construct a valuation $\|-\|$ on
$BP_*BP$ such that $\|v_n\| = \max(0, (6-n)/6)$. They compute $\|\beta\| > 2$ for $\beta \in \mathcal{B}_j$.
Following \cite[after Lemma 11.8]{HHR-v2}, define a valuation on $A[w^\pm]$ by
setting $\|\pi\| = 1/6$ and $\|w\|=0$. Since $\pi^6 = 3u$ for a unit $u$, this
extends the valuation where $\|3\| = 1$. Then analogous arguments as the $p=2$ case show that
$\lambda$ (non-strictly) increases valuation, and so
$\lambda(\beta)=0$ for $\beta\in\mathcal{B}_j$.
\end{proof}



\setcounter{section}{0}
\renewcommand\thesection{\Alph{section}}
\section{Appendix: Supplemental details about the structure of $\hat{H}^*(C_9;
\Sym(\Ind_{C_3}^{C_9}\bar{\rho}))$}
Let $M = \Sym(\Ind_{C_3}^{C_9}\bar{\rho})$.
Lemma \ref{lem:fbar} gives additive generators $1,\bar{c}:=c_1c_2c_3,\delta, \bar{c}\delta = c_1c_2c_3\delta,f_i$, and
$F_i$ (for $0\leq i\leq 5$) for 
$$ \hat{H}^*(C_3; M)^{C_9/C_3} := (\k[d_1,d_2,d_3]\tensor \Lambda_\k[c_1,c_2,c_3])^{C_3} $$
as a module over  
$$S = \k[d_1,d_2,d_3]^{S_3} \isom \k[s_1,s_2,s_3]. $$
Next we give the proof of Proposition \ref{prop:c1-d1}, which completes the
determination of the $S$-module structure on $\hat{H}^*(C_3; M)^{C_3}$ by
finding all $S$-module relations $\sum_i t \cdot \phi$ where $t\in \{
1,\delta,c_1c_2c_3,f_i,F_i \}_{0\leq i\leq 5}$ and $\phi\in S$. The grading
$|c_i| = 1, |d_i| = 0$ on $\hat{H}^*(C_3; M)$, which corresponds to the $q$-grading in
$\SE^{p,q}$, induces a grading on $\hat{H}^*(C_3;M)^{C_3}$ in which
$|\delta| = 0, |f_i| = 1, |F_i| = 2, |\bar{c}| = 3$, and $S$ is in degree 0.

\begin{proof}[Proof of Proposition \ref{prop:c1-d1}]
It is clear from Lemma \ref{lem:c-gens} that in degree 0 and 3 of the $q$-grading,
the $C_3$-fixed points of $\hat{H}^*(C_3; M)$ are free over $S$ on the
generators $\{ 1,\delta \}$ and $\{\bar{c},\bar{c}\delta\}$, respectively.

For degree 1, we need to find relations of the form
\begin{equation}\label{eq:cX-rel} \sum_i f_{n_i}\cdot \phi_i(s_1,s_2,s_3) \end{equation}
where $\phi_i$ is a polynomial in $s_1,s_2,s_3$. The strategy is to turn
this problem about module generators into a problem about ring generators for a
related commutative ring: every relation \eqref{eq:cX-rel} is also true in 
$\k[c_1,c_2,c_3,d_1,d_2,d_3]/(c_i^2)$, as graded commutativity does not affect
such relations. We will use {\tt sage} \cite{sage} to find
\emph{multiplicative} relations in this commutative ring, and then restrict to
relations of $c$-degree 1. Since $\k$ has trivial $C_3$ action, everything is
tensored up over $\F_3$, so we work here over $\F_3$.

\begin{lstlisting}[basicstyle=\small\selectfont\ttfamily,columns=fullflexible]
sage: import sage.libs.singular
....: from sage.rings.polynomial.multi_polynomial_sequence import *
....: _ = var('c1,c2,c3,d1,d2,d3')
....: R.<c1,c2,c3,d1,d2,d3> = PolynomialRing(GF(3))
....: f0 = c1 + c2 + c3
....: f1 = c1*d1 + c2*d2 + c3*d3
....: f2 = c1*d2 + c2*d3 + c3*d1
....: f3 = c1*d1*d2 + c2*d2*d3 + c3*d3*d1
....: f4 = c1*d2*d3 + c2*d3*d1 + c3*d1*d2
....: f5 = c1*d1^2*d2 + c2*d2^2*d3 + c3*d3^2*d1
....: s1 = d1 + d2 + d3
....: s2 = d1*d2 + d2*d3 + d1*d3
....: s3 = d1*d2*d3
....: 
....: S = Sequence([f0, f1, f2, f3, f4, f5, s1, s2, s3])
....: # Find the ideal of algebraic relations among f0, ..., s3
....: I = S.algebraic_dependence()
....: # Find a generating set for the ideal of relations
....: I.groebner_basis()
\end{lstlisting}
The output (list of generating relations) is as follows,
where one should read $(T0,\dots,T8)$ as variables representing
$(f_0,\dots,f_5,s_1,s_2,s_3)$.
\begin{lstlisting}[basicstyle=\small\selectfont\ttfamily,columns=fullflexible]
[T1*T5 - T2*T5 + T3^2 + T3*T4 + T4^2 - T0*T1*T8 + T0*T2*T8 - T0*T4*T7 + T1*T2*T7 - T1*T3*T6 + T0^2*T6*T8,
 T3*T5 - T4*T5 - T0*T3*T8 + T0*T4*T8 - T1^2*T8 - T1*T2*T8 - T2^2*T8 + T2*T3*T7 - T3^2*T6 - T0^2*T7*T8 + T0*T1*T6*T8,
 T5^2 + T0*T5*T8 + T1*T3*T8 - T1*T4*T8 - T2*T3*T8 + T2*T4*T8 + T2*T5*T7 + T3*T4*T7 + T3*T5*T6 + T0^2*T8^2 + T0*T2*T7*T8 + T0*T3*T6*T8 + T2^2*T6*T8 - T2*T3*T6*T7 + T3^2*T6^2,
 T0*T3*T7 - T0*T4*T7 + T0*T5*T6 + T1*T2*T7 - T2^2*T7 - T2*T3*T6 + T2*T4*T6 - T0^2*T6*T8 - T0*T3*T6^2,
 T0*T5*T7 + T1*T4*T7 - T1*T5*T6 - T2*T4*T7 + T2*T5*T6 - T3^2*T6 + T3*T4*T6 - T0^2*T7*T8 + T0*T2*T7^2 - T0*T3*T6*T7 - T1*T2*T6*T7 + T1*T3*T6^2,
 T1^2*T7 + T1*T2*T7 - T1*T3*T6 + T1*T4*T6 + T2^2*T7 + T2*T3*T6 - T2*T4*T6 + T0^2*T7^2 - T0*T1*T6*T7 + T0*T3*T6^2,
 T1*T3*T7 - T1*T4*T7 + T1*T5*T6 - T2*T3*T7 + T2*T4*T7 - T2*T5*T6 - T0*T1*T6*T8 + T0*T2*T6*T8 - T0*T2*T7^2 + T1*T2*T6*T7 - T1*T3*T6^2,
 T2*T5*T7 - T3*T4*T7 + T4^2*T7 - T4*T5*T6 - T0*T2*T7*T8 + T0*T4*T6*T8 - T1*T2*T6*T8 + T2^2*T6*T8 + T2^2*T7^2 - T2*T3*T6*T7 - T2*T4*T6*T7 + T3*T4*T6^2 + T0*T2*T6^2*T8,
 T3^3 - T4^3 + T0*T4*T5*T6 + T1^3*T8 - T1^2*T5*T6 - T1*T2*T5*T6 - T1*T3^2*T6 - T1*T4^2*T6 - T2^3*T8 + T2^2*T3*T7 - T2^2*T4*T7 - T2^2*T5*T6 - T2*T4^2*T6 + T0^2*T1*T7*T8 - T0^2*T2*T7*T8 + T0^2*T3*T6*T8 + T0^2*T3*T7^2 + T0^2*T4*T6*T8 - T0^2*T4*T7^2 - T0*T1*T2*T6*T8 + T0*T1*T2*T7^2 - T0*T1*T3*T6*T7 + T0*T1*T4*T6*T7 + T0*T2^2*T6*T8 - T0*T2^2*T7^2 + T0*T3^2*T6^2 + T0*T3*T4*T6^2 - T1^2*T2*T6*T7 + T1^2*T3*T6^2 + T1*T2^2*T6*T7 - T1*T2*T3*T6^2,
 T3^2*T7 + T3*T4*T7 + T4^2*T7 + T1^2*T6*T8 + T1*T2*T6*T8 + T2^2*T6*T8 + T2^2*T7^2 - T2*T3*T6*T7 - T2*T4*T6*T7 + T3*T4*T6^2 + T0^2*T6*T7*T8 - T0*T1*T6^2*T8 + T0*T2*T6^2*T8]
\end{lstlisting}
The relations are homogeneous with respect to the
$q$-grading (i.e., $|d_i| = 0, |c_i| = 1$), and there are no relations in homogeneous degree 1. Any relation of the form \eqref{eq:cX-rel}
in the graded commutative ring would also be a relation in the ordinary
polynomial-exterior ring, and so $\hat{H}^*(C_3; M)$ restricted to $q$-degree 1 is free over $S$ on the generators $f_0,\dots,f_5$.

The statement for $c$-degree 2 is formally the same as the first,
where $c_1c_2, c_2c_3, c_3c_1$ can be thought of as indecomposable variables
playing the role of $c_1,c_2,c_3$.
\end{proof}

\begin{lemma}\label{lem:extra-relations}
We have the following additional multiplicative relations in $\hat{H}^*(C_9; \Sym(\Ind_{C_3}^{C_9}
\bar{\rho}))$.
\begin{align*}
\delta^2  &=  - s_2^3 - s_1^3 s_3 + s_1^2 s_2^2
\\f_0f_1 & = F_2 - F_1
\\f_0f_2 & = F_0 s_1 - F_1 + F_2
\\f_0f_3 & = F_4 - F_3
\\f_0f_4  & = F_0 s_2 - F_3 + F_4
\\f_0f_5 & = F_2s_2 - F_3 s_1
\\f_1f_2 & = F_0 s_2 - F_2 s_1
\\f_1 f_3 & = F_5 - F_0 s_3 - F_3 s_1
\\f_1 f_4 & = F_5 - F_0 s_3 + F_3 s_1 + F_1s_2 - F_2 s_2
\\f_1f_5 & = F_1 s_3 - F_2 s_3 - F_3 s_2
\\f_2 f_3  & = F_5 + F_2 s_2 - F_3 s_1 - F_0 s_3
\\f_2f_4  & = F_5 - F_0 s_3 + F_2 s_2 - F_3 s_1 - F_4 s_1
\\f_2 f_5  & = F_0 s_1s_3 + F_1s_3 + F_2 s_1s_2 - F_3 s_1^2 - F_4 s_2 + F_5 s_1
\\f_3 f_4 & = F_0 s_1s_3 - F_4 s_2
\\f_3 f_5  & =  F_3 s_3 - F_4 s_3 + F_2 s_1s_3
\\f_4f_5 & = F_0s_2s_3 - F_1s_1s_3+F_2s_2^2-F_2s_1s_3+F_3s_3-F_3s_1s_2-F_4s_3+F_5s_2
\\f_0\delta  & = -f_0 s_1s_2 - f_1s_2 + f_2s_2 + f_3s_1 - f_4s_1
\\ f_1\delta  & = f_0s_2^2 - f_0s_1s_3 + f_1s_1s_2 +  f_3 s_2 - f_4 s_2 + f_5 s_1
\\ f_2\delta  & = - f_0 s_1s_3 - f_2 s_1s_2 + f_3 s_2 - f_3 s_1^2 - f_4 s_2 + f_5 s_1
\\ f_3\delta  & = - f_0 s_2s_3 - f_1 s_1 s_3 + f_2 s_1s_3 + f_3 s_1s_2 + f_5 s_2
\\ f_4\delta  & = -f_0 s_2 s_3 + f_0 s_1^2s_3 - f_1 s_1 s_3 + f_2 s_2^2 + f_2 s_1s_3 - f_3 s_1 s_2 + f_4 s_1s_2 + f_5 s_2
\\ f_5\delta  & = f_0 s_1s_2 s_3 - f_1s_2 s_3 - f_1s_1^2s_3 + f_2 s_2 s_3 - f_3 s_2^2 + f_3 s_1 s_3 - f_4 s_1 s_3 - f_5 s_1s_2
\end{align*}
\end{lemma}
\begin{proof}
These were all verified explicitly using {\tt sage}, version 10.6.
Even though $\hat{H}^*(C_9;
\Sym(\Ind^{C_9}_{C_3}\bar{\rho}))$ has ground ring $\W/9$, we may work over $\k$
(and hence over $\F_3$) here since all of the elements involved in these relations
are 3-torsion due to the relations in Theorem \ref{thm:main-completion}.

To check the relations, one can use the following setup. (We just show the first relation but the rest are similar.)
\begin{lstlisting}[basicstyle=\small\selectfont\ttfamily,columns=fullflexible]
sage: Gr.<c1,c2,c3,d1,d2,d3> = GradedCommutativeAlgebra(GF(3), degrees=(1,1,1,2,2,2))
....: def tr(f):
....:     return (f +
....:             f.substitute({c1:c2,c2:c3,c3:c1,d1:d2,d2:d3,d3:d1}) +
....:             f.substitute({c1:c3,c2:c1,c3:c2,d1:d3,d2:d1,d3:d2}))
....: f0 = tr(c1)
....: f1 = tr(c1*d1)
....: f2 = tr(c1*d2)
....: f3 = tr(c1*d1*d2)
....: f4 = tr(c1*d2*d3)
....: f5 = tr(c1*d1^2*d2)
....: F0 = tr(c1*c2)
....: F1 = tr(c1*c2*d1)
....: F2 = tr(c1*c2*d2)
....: F3 = tr(c1*c2*d1*d2)
....: F4 = tr(c1*c2*d2*d3)
....: F5 = tr(c1*c2*d1^2*d2)
....: s1 = tr(d1)
....: s2 = tr(d1*d2)
....: s3 = d1*d2*d3
....: delta = (d1-d2)*(d2-d3)*(d3-d1)
sage: delta^2 == -s2^3 - s1^3*s3 + s1^2*s2^2
True
\end{lstlisting}
\end{proof}

\bibliographystyle{alpha}
\bibliography{bib.bib}

\end{document}